\documentclass{amsproc}
\usepackage{euscript}
\usepackage{cases}
\usepackage{mathrsfs}
\usepackage{bbm}
\usepackage{amssymb}
\usepackage{txfonts}
\usepackage{amsfonts,amsmath,amsxtra,mathdots,mathabx}
\usepackage{color}
\usepackage{hyperref}
\usepackage{tikz}

\allowdisplaybreaks

\DeclareFontFamily{U}{matha}{\hyphenchar\font45}
\DeclareFontShape{U}{matha}{m}{n}{
	<5> <6> <7> <8> <9> <10> gen * matha
	<10.95> matha10 <12> <14.4> <17.28> <20.74> <24.88> matha12
}{}
\DeclareSymbolFont{matha}{U}{matha}{m}{n}

\DeclareMathSymbol{\Lt}{3}{matha}{"CE}
\DeclareMathSymbol{\Gt}{3}{matha}{"CF}

\DeclareSymbolFont{mathc}{OML}{txmi}{m}{it}
\DeclareMathSymbol{\varvv}{\mathord}{mathc}{118}
\DeclareMathSymbol{\varnu}{\mathord}{mathc}{"17}

\DeclareSymbolFont{mathd}{OML}{ztmcm}{m}{it}
\DeclareMathSymbol{\varalpha}{\mathord}{mathd}{11}
\def\valpha{\scalebox{0.87}{$\varalpha$}}
\def\salpha{\scalebox{0.62}{$\varalpha$}}

\DeclareMathSymbol{\vvepsilon}{\mathord}{mathd}{15}
\def\vepsilon{\scalebox{0.87}{$\vvepsilon$}}
\def\sepsilon{\scalebox{0.62}{$\vvepsilon$}}
\def\ssepsilon{\scalebox{0.5}{$\vvepsilon$}}

\DeclareMathSymbol{\varchi}{\mathord}{mathd}{31}
\def\vchi{\scalebox{0.9}{$\varchi$}}

 \newcommand{\BZ}{{\mathbb {Z}}}

\def\CalJ{\text{\usefont{U}{dutchcal}{m}{n}J}\hskip 0.5pt}
\def\CalL{\text{\usefont{U}{dutchcal}{m}{n}L}\hskip 0.5pt}
\def\CalK{\text{\usefont{U}{dutchcal}{m}{n}K}\hskip 0.5pt}

\newcommand{\GL}{{\mathrm {GL}}}

\newcommand{\sstyle}{\scriptstyle}

\newcommand{\ra}{\rightarrow}

\def\-{^{-1}}

\def\sasymp{\text{ \small $\asymp$ }}
\def\mod{\mathrm{mod}\, }

\def\sumx{\sideset{}{^\star}\sum}

\def\tw{\textit{w}}
\def\nd{\mathrm{d}}

\def\lp {\left (}
\def\rp {\right )}

\def\Voronoi{Vorono\" \i \hskip 2.5 pt }

\renewcommand{\Im}{{\mathrm{Im} }}
\renewcommand{\Re}{{\mathrm{Re} }}

\def\shskip{\hskip 1pt}

\makeatletter
\g@addto@macro\normalsize{\setlength\abovedisplayskip{3pt}}
\makeatother

\makeatletter
\g@addto@macro\normalsize{\setlength\belowdisplayskip{3pt}}
\makeatother

\newcommand{\delete}[1]{}

\theoremstyle{plain}

\newtheorem{thm}{Theorem}[section] \newtheorem{cor}[thm]{Corollary}
\newtheorem{lem}[thm]{Lemma}  \newtheorem{prop}[thm]{Proposition}

\newtheorem {rem}[thm]{Remark}

\newtheorem*{acknowledgement}{Acknowledgements}

\numberwithin{equation}{section}

\begin{document}

	\title{A Bessel delta-method and exponential sums for $\mathrm{GL} (2)$}
	\author[K. Aggarwal, R. Holowinsky, Y. Lin, and Z. Qi]{Keshav Aggarwal, Roman Holowinsky, Yongxiao Lin, and Zhi Qi}
	\address{Department of Mathematics and Statistics, The University of Maine, 5752 Neville Hall, Room 333 Orono, ME 04469, USA}
	\email{keshav.aggarwal@maine.edu}
	\address{Department of Mathematics, The Ohio State University\\ 231 W 18th Avenue\\
		Columbus, Ohio 43210, USA}
	\email{holowinsky.1@osu.edu}
	\address{EPFL SB MATHGEOM TAN, Station 8, CH-1015, Lausanne, Switzerland}
	\email{yongxiao.lin@epfl.ch}
	\address{School of Mathematical Sciences, Zhejiang University, Hangzhou, 310027, China}
	\email{zhi.qi@zju.edu.cn}

	\begin{abstract}
		In this paper, we introduce a simple Bessel $\delta$-method to the theory of exponential sums for $\rm GL_2$. Some results of Jutila on exponential sums are generalized in a less technical manner to holomorphic newforms of arbitrary level and nebentypus. In particular, this gives a short proof for the Weyl-type subconvex  bound in the $t$-aspect for the associated $L$-functions.
	\end{abstract}

	
	\maketitle

	\section{Introduction}

	It is a classical problem to estimate exponential sums involving the Fourier coefficients of a modular form. 
	Let $g \in S^{\star}_k (M, \xi)$ be a holomorphic cusp newform of level $M$, weight $k$, nebentypus character $\xi$, with the Fourier expansion
	$$g (z) = \sum_{n=1}^{\infty} \lambdaup_g (n) n^{(k-1)/2} e (n z), \quad e(z) = e^{2 \pi i z} ,$$
	for  $\Im \, z > 0$. For example, it is well-known that for any real $\gamma$ and $N \geqslant 1$,
	\begin{align}\label{0eq: holomorphic form}
		\sum_{n \shskip \leqslant N} \lambdaup_g (n) e (\gamma n) \Lt_g N^{1/2} \log (2 N),
	\end{align}
	with the implied constant depending only on $g$ (see  \cite[Theorem 5.3]{Iwaniec-Topics}). This is a classical estimate due to Wilton. This type of estimates with uniformity in $\gamma$ was generalized by Stephen D. Miller to cusp forms for $\GL_3 (\BZ)$ in \cite{Miller-Wilton}. 
	
	In this paper,  we consider the following exponential sum (and its variants), 
	\begin{align}\label{0eq: S sharp}
		S^{\scriptscriptstyle \sharp}(N) = S_{f }^{\scriptscriptstyle \sharp}(N) =\sum_{N  \leqslant n\shskip \leqslant 2N}\lambdaup_g(n)\,e(f  (n) ),  
	\end{align}
	for a phase function $f  $ of the form  
	\begin{align}\label{0eq: f = T phi}
		f  (x) = T \phi (x/ N) + \gamma x,
	\end{align}
	where $\phi  $ is real-valued and smooth (see  Theorem \ref{main-theorem}),  $\gamma $ is real, and  $N, T \geqslant 1$ are large parameters. We assume here that $\phi  $ is {\it not} a linear function, as otherwise the sum is already estimated in \eqref{0eq: holomorphic form}. 
	As usual, we shall be mainly investigating the smoothed exponential sum 
	\begin{align}
		S (N) =  S_{f } (N) = \sum_{n=1}^\infty \lambdaup_g(n) e(f  (n) ) V\left(\frac{n}{N}\right),
	\end{align}
	for a certain smooth weight function $V \in  C_c^{\infty} (0, \infty) $ supported in $[1, 2]$ as described in Theorem \ref{main-theorem}. 
	
	This type of exponential sums (with $\gamma = 0$) for modular forms $g$ of level $M = 1$ were first studied by Jutila \cite{Jut87}, using  Farey fractions,  the \Voronoi summation formula, and stationary phase analysis. See also \cite[\S 10]{Huxley}  for an account of Jutila's method.

	Thanks to the Rankin--Selberg theory, we know that  $|\lambdaup_g (n)|$'s obey the Ramanujan conjecture on average:
	\begin{align}\label{2eq: Ramanujan}
		\sum_{n \shskip \leqslant N} |\lambdaup_g (n)|^2 \Lt_{g} N. 
	\end{align}  
	Moreover, by the work of  Deligne \cite{Deligne} and Deligne--Serre \cite{Deligne-Serre} (the latter is for $k=1$), the  Ramanujan conjecture for holomorphic cusp forms is now well-known:
	\begin{align}\label{2eq: Ramanujan, 0}
		\lambdaup_g (n) \Lt n^{\sepsilon }.
	\end{align} 
	An application of the Cauchy--Schwarz inequality followed by \eqref{2eq: Ramanujan} yields the trivial bounds $S(N)$, $  S^{\scriptscriptstyle \sharp}(N)\Lt_{ g} N$. Thus one aims to improve over these trivial bounds or, in other words, to show that there is no correlation between $ \lambdaup_g (n) $ and $e (f(n))$.

	The primary purpose of this paper is to find a $\delta$-method which is analytically richer so that the stationary phase analysis at later stages becomes cleaner. It turns out that an added benefit of our pursuit is a generalization of some results in Jutila's treatise \cite{Jut87} to modular forms of arbitrary level and nebentypus. An application amongst others is the Weyl-type subconvex bound for the associated $L$-functions in the $t$-aspect.  
	
	The main novelty of our work is a simple Bessel $\delta$-method to be described as follows. 
	
	\subsection*{A simple Bessel $\delta$-method}

	As usual, let $e (x) = e^{ 2\pi i x}$ and let $J_{\varnu} (x)$ be the $J$-Bessel function of order $\varnu$. For a condition $\mathrm{C}$, let  $\delta (\mathrm{C})$ denote the Kronecker $\delta$ that detects $\mathrm{C}$.	
	
	We fix a smooth bump function $U$ in $C_c^{\infty}(0, \infty)$. 
	Our Bessel  $\delta$-method is based on the observation that for a prime $p$, some large parameters $N, X $, and integers $r, n \in [N, 2N]$, one has
	\begin{equation*}
		\begin{split}
			&\hskip 12.5pt \frac 1 p \sum_{a (\mod p)}e\left(\frac{a(n-r)}{p}\right)\cdot \int_0^\infty e\left(\frac{2\sqrt{rx}}{p}\right)J_{k-1}\left(\frac{4\pi\sqrt{nx}}{p}\right)U\left(\frac{x}{X}\right)\mathrm{d}x\\
			&= \delta ({r\equiv n  (\mod p)}) \cdot \delta \big(|r-n| < X^{ \sepsilon} p {\textstyle \sqrt {N/X}} \big) \cdot \text{``some\, factor''}+\text{``error"}\\
			&=\delta ({r=n}) \cdot \text{``some\, factor''}+\text{``error"},
	\end{split}\end{equation*}
	provided that $N < X^{1-\sepsilon}$ and $p^2 < N X$. 
	This is made  explicit in Lemma \ref{lem: Bessel-delta}. The merit of this Bessel $\delta$-identity is that it arises naturally from the \Voronoi summation formula, for  the Bessel integral may be interpreted as the Hankel transform of  $ e\left({2\sqrt{rx}} / {p}\right) U (x/X) $. 
	
	As explained in \S \ref{sec:Weber-Hankel}, there is a vague but interesting connection between the Bessel integral above and the formula
	\begin{align*}
		\int_0^{\infty}  J_{k-1} ( 4 \pi a \hskip -1pt \sqrt x)  J_{k-1} ( 4 \pi b \hskip -1pt  \sqrt x) \shskip \nd x = \frac {   \delta (a-b) } { 8 \pi^2 \hskip -1pt b },
	\end{align*}
	where $\delta (a-b)$ is now the Dirac $\delta$-distribution. Thus the use of $\delta$ is justified from a different perspective.

	\subsection*{Main results}


	\begin{thm}\label{main-theorem}
		Let $\vepsilon  > 0$ be an arbitrarily small constant.	
		Let $N, T, \varDelta > 1$ be parameters such that
		\begin{align}\label{1eq: condition on Delta}
			N^{\sepsilon} \varDelta \leqslant T .
		\end{align}
		Let $V (x) \in C_c^{\infty} (0, \infty) $ be a smooth function with support in $[1, 2]$. Assume that its total variation $\mathrm{Var} (V) \Lt 1$ and that $V^{(j)} (x) \Lt_{j} \varDelta^j$  for $j \geqslant 0$. For $\gamma$ real, and $ \phi (x) \in C^{\infty} (1/2, 5/2) $ satisfying   $ |\phi'' (x)| \Gt 1 $ and  $  \phi^{(j)} (x) \Lt_j 1 $ for $j \geqslant 1$, define $f (x) = T \phi (x/ N) + \gamma x$. Let $g \in S^{\star}_k (M, \xi)$ and $ \lambdaup_g(n) $ be its Fourier coefficients.  Then
		\begin{equation}\label{1eq: main bound}
			\sum_{n=1}^\infty \lambdaup_g(n) e ( f(n) ) V\left(\frac{n}{N}\right) \Lt  T^{1/3} N^{  1 / 2  +\sepsilon} +  \frac { N^{  1 + \sepsilon} }  { T^{  1/6}} ,
		\end{equation}
		with the implied constant depending only on $g$,  $\phi$ and $\vepsilon$. 
	\end{thm}

	\begin{cor}\label{main-corollary}
		Let $\phi$, $f$ and $g$ be as above. Let $    N^{1+\sepsilon}/ T \leqslant H \leqslant N$.	 We have
		\begin{align}\label{1eq: main bound, 0.2}
			\sum_{N \leqslant n \shskip \leqslant N+H}\lambdaup_g(n)\,e(f(n))\Lt_{g,  \shskip \phi, \shskip \sepsilon}   T^{1/3} N^{  1 / 2  +\sepsilon} +  \frac { N^{  1 + \sepsilon} }  { T^{  1/6}} .
		\end{align} 
		As a consequence, 
		\begin{align}\label{1eq: main bound, 0.3}
			S_{f }^{\scriptscriptstyle \sharp}(N) =	\sum_{N \leqslant n \shskip \leqslant 2N}\lambdaup_g(n)\,e(f(n))\Lt_{g,  \shskip \phi, \shskip \sepsilon}   T^{1/3} N^{  1 / 2  +\sepsilon} +  \frac { N^{  1 + \sepsilon} }  { T^{  1/6}} .
		\end{align} 
	\end{cor}	
	
	Jutila's estimate for $ S_{f }^{\scriptscriptstyle \sharp}(N) $, say, for modular forms $g$ of level $M=1$ and for phase functions $f (x) = T \phi (x/N)$ (see  \cite[\S 10]{Huxley}) is as follows,
	\begin{align}
		S_{f }^{\scriptscriptstyle \sharp}(N) \Lt_{g,  \shskip \phi, \shskip \sepsilon}   T^{1/3} N^{  1 / 2  +\sepsilon},
	\end{align}
	provided that $ N^{3/4} < T < N^{3/2} $.
	
	Corollary \ref{main-corollary} may be regarded as a generalization of Theorem 4.6 of Jutila \cite{Jut87} in several aspects. First of all, the modular form $g $ here is of arbitrary level and nebentypus. Secondly, the estimate in \eqref{1eq: main bound, 0.3} is non-trivial as long as $N^{\sepsilon} < T < N^{3/2-\sepsilon}$, while it is assumed in \cite{Jut87} that $ N^{3/4} < T < N^{3/2} $. Note that our estimate is weaker than Jutila's when $ N^{3/4} < T < N $. Nevertheless, we are usually more concerned with the case when $ N^{1-\sepsilon} < T < N^{3/2-\sepsilon} $, for example, in the subconvexity problem; our estimate is the same as Jutila's in this case. Thirdly, our phase function $f (x)$ contains an additional linear term $\gamma x$. 
	
	Note that we shall be content with the averaged Ramanujan conjecture \eqref{2eq: Ramanujan} in the proof of  Theorem \ref{main-theorem}, while we shall need the Ramanujan conjecture (Deligne's bound) \eqref{2eq: Ramanujan, 0} only for the deduction of Corollary \ref{main-corollary}.
	
	For ease of exposition, only holomorphic modular forms are considered here, but our approach also works for Maass forms with some efforts. 
	
	\subsection*{Examples} 
	A typical and simple choice of $\phi (x)$ is the power function $ \pm x^{\shskip\beta}$ so that $ f (x) = \valpha x^{\shskip\beta} + \gamma x$ ($T = |\valpha| N^{\beta}$). Let 
	\begin{align*}
		S_{\salpha, \shskip \beta, \shskip \gamma}^{\scriptscriptstyle \sharp} (N) = \sum_{n \shskip \leqslant N}\lambdaup_g(n)\,e(\valpha n^{\shskip \beta} + \gamma n). 
	\end{align*}
	
	For modular forms $g$ of level $M=1$, there are abundant works on this type of exponential sums in the literature (usually, with $\gamma = 0$). 
	
	As alluded to above,	the first non-trivial  bound  for $S_{\salpha, \shskip \beta, \shskip 0}^{\scriptscriptstyle \sharp} (N)$ was obtained by Jutila (see  \cite[Theorem 4.6]{Jut87}) for the range $3/4 < \beta < 3/2$, $\beta \neq 1$, as follows,
	\begin{align}\label{0eq: Jutila}
		S_{\salpha, \shskip \beta, \shskip 0}^{\scriptscriptstyle \sharp}(N)\Lt_{g, \shskip \salpha, \shskip \beta, \shskip \sepsilon} N^{\frac 1 2 + \frac {\beta} 3 +\sepsilon } .
	\end{align}

	When $\beta=1/2$, $\valpha= - 2\sqrt{q}$ for integer $q > 0$, and $\gamma = 0$, it was first shown by Iwaniec, Luo  and Sarnak \cite[(C.17)]{ILS-LLZ} that the smoothed sum $$\sum_{n=1}^\infty \lambdaup_g(n) e(-2\sqrt{qn} ) V\left(\frac{n}{N}\right)$$ has a main term of size $N^{3/4}$.
	
	The first non-trivial bound towards $S_{\salpha, \shskip \beta, \shskip \gamma}^{\scriptscriptstyle \sharp} (N)$ for {\it all}  $0 < \beta < 1$ is due to X. Ren and Y. Ye \cite{RenYe}, who refined the aforementioned result of  Iwaniec, Luo and Sarnak for $\beta = 1/2$, and proved for $\beta \neq 1/2$ that
	\begin{align}\label{0eq: Ren-Ye}
		S_{\salpha, \shskip \beta, \shskip 0}^{\scriptscriptstyle \sharp}(N)\Lt_{g, \shskip \salpha, \shskip \beta, \shskip \sepsilon} N^{\beta +\sepsilon }+N^{\frac 1 2-\frac \beta 4+\sepsilon}. 
	\end{align}
	This was improved into $N^{1/3+\sepsilon}$ in \cite{Sun-Wu} for $0 < \beta < 1/2$ (the Maass form case is also considered there).  Note that Jutila's estimate \eqref{0eq: Jutila} is stronger than \eqref{0eq: Ren-Ye} for $3/4 < \beta < 1$. 
	
	It should be mentioned that Q. Sun \cite{Sun-1} obtained the bound $N^{1- \frac {\beta} 2 + \sepsilon}$ for $ S_{\salpha, \shskip \beta, \shskip \gamma}^{\scriptscriptstyle \sharp}(N)  $ in the range $ 0 < \beta \leqslant 1/2$.  Her bound was improved into $N^{\frac 1 2 + \frac {\beta} 2+ \sepsilon}$ by Godber \cite{Additive-Godber} (for $0<\beta< 1$). For $\gamma = 0$, these bounds are both weaker than \eqref{0eq: Ren-Ye}.
	
	There is also a very distinguishable result---Pitt's  {\it uniform} estimate for $S_{\salpha, \shskip 2, \shskip \gamma}^{\scriptscriptstyle \sharp} (N)$ with quadratic phase in \cite{Pitt-quadratic},
	\begin{align}\label{0eq: quadratic}
		S_{\salpha, \shskip 2, \shskip \gamma}^{\scriptscriptstyle \sharp} (N) \Lt_{g, \shskip \sepsilon} N^{\frac {15}  {16} + \sepsilon},
	\end{align}
	where the implied constant depends only on $g$ and $\vepsilon$.  The exponent $15/16$ was later improved into $7/8$ by K. Liu and X. Ren \cite{Ren-Liu-quadratic}.
	
	More generally, one can also consider analogous exponential sums of Fourier coefficients of  Maass cusp forms for $\GL_m$, $m\geqslant 3$.   Some similar results for $\GL_3$ and $\GL_m$ were obtained later by X. Ren and Y. Ye in \cite{RenYe2,RenYe3}. Recently, Kumar et al. \cite{Kumar-M-S} had some improvement over the results in \cite{RenYe2}, by using the $\delta$-symbol method of Duke--Friedlander--Iwaniec \cite{DFI-1} together with a conductor-lowering trick which was first introduced by Munshi \cite{Munshi-Circle-III}.

	A direct consequence of Corollary \ref{main-corollary} is the following estimates for $  S_{\salpha, \shskip \beta, \shskip \gamma}^{\scriptscriptstyle \sharp} (N)$ for modular forms $g \in  S^{\star}_k (M, \xi)$. 
	
	\begin{cor}
		Let $g \in S^{\star}_k (M, \xi)$ and $ \lambdaup_g(n) $ be its Fourier coefficients.		For real $\valpha, \gamma$ and $\beta$ with $\valpha \neq 0$, $\beta \neq 1$, we have 
		\begin{align}\label{1eq: main bound, 2}
			\sum_{n \shskip \leqslant N}\lambdaup_g(n)\,e(\valpha n^{\shskip \beta} + \gamma n)\Lt_{g,  \shskip \beta, \shskip \sepsilon}   |\valpha|^{\frac 1 3} N^{\frac 1 2+ \frac \beta 3 +\sepsilon} +     {|\valpha|^{- \frac 16} } {N^{1 -\frac {\beta} 6 + \sepsilon}  } .
		\end{align} 
		In particular, 
		\begin{align}\label{1eq: main bound, 3}
			\sum_{n \shskip \leqslant N}\lambdaup_g(n)\,e(\valpha n^{\shskip \beta}  + \gamma n)\Lt_{g, \shskip \salpha, \shskip \beta, \shskip \sepsilon}  N^{\frac 1 2+ \frac \beta 3 +\sepsilon} +  {N^{1-\frac {\beta} 6 + \sepsilon}  } .
		\end{align} 
	\end{cor}
	
	Note that the estimate  \eqref{1eq: main bound, 3} is non-trivial for $0 < \beta < 3/2 $. Though weaker for $3/4 < \beta < 1$, it is the same as Jutila's estimate \eqref{0eq: Jutila} for $1 < \beta < 3/2$. At any rate, our estimate is an extension of Jutila's result (for $1 < \beta < 3/2$, literally) to modular forms of general level.
	
	Also note that \eqref{1eq: main bound, 3} is better than Ren and Ye's estimate \eqref{0eq: Ren-Ye}  as long as $\beta> 6/7$. 
	However, our bound   is worse than theirs for $\beta < 6/7$. This is due to the nature and the limitation of our Bessel $\delta$-method or any $\delta$- or circle method. For if $ \beta  $ is relatively small then $e(\valpha n^{\shskip \beta})$ is not quite oscillatory, and  it would not benefit much to separate the oscillations of $e(\valpha n^{\shskip \beta})$ and $\lambdaup_g (n)$ by the $\delta$-method.  The approach in \cite{RenYe} works far better in this situation, where the \Voronoi summation (with modulus $1$) is applied directly,  followed by stationary phase arguments. 
	
	In \cite{Pitt-quadratic}, the $\delta$-method of Duke--Friedlander--Iwaniec, along with Diophantine approximation,  is used to prove the estimate in \eqref{0eq: quadratic} in the quadratic case $\beta = 2$. However, this approach does not work with fractional $\beta$.

	\subsection*{Application: Weyl-type subconvex bound in the $t$-aspect}
	
	For $g \in S^{\star}_k (M, \xi)$ with Fourier coefficients
	$\lambdaup_g(n)$, the associated $L$-function is given by
	\begin{equation*}
		\begin{split}
			L(s,g) =\sum_{n=1}^{\infty}\frac{\lambdaup _g(n)}{n^s}, \hskip 15pt \Re\, s >1.
	\end{split}\end{equation*}
	This $L$-series has an analytic continuation  to the whole complex plane. The Phragm\'en--Lindel\"of principle implies the $t$-aspect convex bound $$ L\left(1/2+it,g\right)\Lt_{g, \shskip \sepsilon} (1+|t|)^{1/2+\sepsilon}$$ for any $\vepsilon>0$. Any improvement on the exponent on the right-hand side of the inequality is referred to as a subconvex  bound, and in general it requires significant  amount of work to achieve it.
	
	When $M =1$, the following Weyl-type subconvex bound was first proven by Good \cite{Good},
	$$L(1/2+it,g)\Lt_{g, \shskip \sepsilon} (1+|t|)^{1/3+\sepsilon}, $$
	by appealing to the spectral theory of automorphic functions. Later, the same bound was obtained by Jutila using his method developped in \cite{Jut87}. See \cite{Meu87,Jutila97} for the extension of these methods to the Maass-form case.  
	
	There has been much progress lately, due to new methods, especially variants of the $\delta$-symbol or circle method become available. For example, Munshi \cite{Munshi-Circle-III} solved the $t$-aspect subconvexity problem for  $L$-functions on $\rm GL_3$ by adopting Kloosterman's version of the circle method.  
	He also invented  a $\rm GL_2$ $\delta$-method and used it in a series of papers \cite{Munshi-Circle-IV}--\cite{Munshi17.5} for various subconvexity problems. These methods were applied in \cite{Aggarwal-Singh2017,ARSS} to obtain the Weyl bound in the $\GL_2$ setting.
	In a recent preprint \cite{Munshi2018-6}, Munshi was even able to break the long standing Weyl-bound barrier by introducing extra variants into the $\rm GL_2$ $\delta$-method approach. 
	
	Recently, there are Weyl-type subconvexity results for cusp forms of general level by Booker et al.  \cite{BMN} and the first-named author \cite{Aggarwal}. Booker et al. \cite{BMN} generalized Huxley's treatment of Jutila's method by using a \Voronoi  formula with {\it arbitrary} additive twists to obtain their result. On the other hand, Aggarwal \cite{Aggarwal} used a simple $\delta$-symbol method and followed Munshi's approach \cite{Munshi-Circle-III}. This treatment allowed him to use the \Voronoi formula of Kowalski--Michel--VanderKam to get the Weyl-type bound, along with an explicit dependence on the level of the cusp form.

	By applying Theorem \ref{main-theorem}, with $\phi (x) = - \log x$, we shall derive in \S \ref{sec: Weyl} the Weyl subconvex bound for $ g \in S^{\star}_k (M, \xi)$.
	
	\begin{thm}\label{main-theorem-Weyl}
		Let $g \in S^{\star}_k (M, \xi)$. Then
		\begin{equation*}
			L\left(1/2+it,g\right)\Lt  (1+|t|)^{1/3+\sepsilon}.
		\end{equation*}
		with the implied constant depending only on $g$ and $\vepsilon$. 
	\end{thm}
	
	This work is of the same theme as  \cite{Aggarwal}, but it is technically simpler here, for our Bessel $\delta$-method is more intimate to the \Voronoi summation formula than his trivial $\delta$-method.
	Moreover, our argument by the Bessel $\delta$-method is very short compared to that by the Jutila method generalized in \cite{BMN}. 
	
	A motivation of our work is from \cite{AHLS}, in which, together with Q. Sun, the first three named authors investigated subconvex bounds for $L(1/2, g\otimes\vchi)$, where $\vchi$ is a primitive Dirichlet character of prime conductor $q$. They were able to use a `trivial' delta method to give a simpler proof for the Burgess bound  in the $q$-aspect, 
	$$L(1/2, g\otimes\vchi)\Lt_{g, \shskip \sepsilon} q^{3/8+\sepsilon}. $$ 
	The Bessel $\delta$-method is an outcome of our search for a similar simple approach to strong subconvex bounds in the $t$-aspect. 
	It seems natural that the argument of this paper can be combined with the approach in \cite{AHLS} to obtain a uniform subconvexity bound for $L(1/2+it,g\otimes \vchi)$ in both the $q$ and $t$ aspects.
	
	\subsection*{Notation} Let $p$ always stand for prime. The notation $n \sim N$ or $p \sim P$ is used for integers or primes in the dyadic segment $[N, 2N]$ or $[P, 2P]$, respectively.


	\section{\texorpdfstring{The \Voronoi summation}{The Voronoi summation}}
	

	Let $S^{\star}_k (M, \xi)$ denote the set of primitive newforms of level $M$, weight $k$ and nebentypus $\xi$. We have necessarily $\xi (-1) = (-1)^k$.    The term ``primitive" means that the form is Hecke-normalized so that its  Fourier coefficients and Hecke eigenvalues coincide.

	The following \Voronoi summation formula is a special case of \cite[Theorem A.4]{KMV}. Note that $g_M = \widebar {g} \in S^{\star}_k (M, \widebar \xi)$ in their notation (see  \cite[Proposition A.1]{KMV}). 
	
	\begin{lem}[The \Voronoi Summation Formula]\label{lem: Voronoi}
		Let $g$ be a primitive holomorphic newform in $S^{\star}_k (M, \xi)$. Let $a, \overline{a}, c$ be integers such that $c \geqslant 1$, $(a, c) = 1$, $a \overline{a} \equiv 1 (\mod c)$ and $(c, M) = 1$. Let $F (x) \in C_c^{\infty} (0, \infty)$. Then there exists a complex number $\eta_g $ of modulus $1$ {\rm(}the Atkin--Lehner pseudo-eigenvalue of $g${\rm)} such that
		\begin{equation}\label{1eq: Voronoi}
			\begin{split}
				\sum_{n=1}^{\infty}  \lambdaup_g(n)  e\left(\frac{an}{c}\right) F\left( {n} \right) =  \frac {\eta_g \xi (-c)} {c \sqrt{M}} \sum_{n=1}^{\infty} \overline { \lambdaup_{  g } (n) } e\left(- \frac{\overline{a}n}{c}\right) \check{F} \lp \frac{  n }{c^2 M} \rp .
			\end{split}
		\end{equation} 
		where $\check{F} (y)$ is the Hankel transform of $F (x)$ defined by 
		\begin{align}\label{2eq: Hankel transform}
			\check{F} (y) = 2\pi i^k\int_0^\infty F(x) J_{k-1}\left( {4\pi\sqrt{xy}} \right) \hskip -1pt \mathrm{d}x
		\end{align}
	\end{lem}
	
	

	The \Voronoi summation formula in \cite[Theorem A.4]{KMV} is more general, where it is only required that $((c, M), M/(c,M) ) = 1$. 
	However, in our setting $c = p$ will be a large prime while $M$ is fixed, so our condition $(c, M) = 1$ in Lemma \ref{lem: Voronoi} is justified. For comparison, we remark that, Jutila's method requires the $a/c$ to be {\it every} fraction, so this \Voronoi works only if $M$ is square-free; thus in  \cite{BMN}, they need a more general \Voronoi even without the restriction $((c, M), M/(c,M) ) = 1$.

	\section{A Bessel $\delta$-method}

	\subsection{Basics of Bessel functions}
	
	For complex $\varnu$, let $J_{\varnu} (z)$ 
	be the Bessel function of the first kind (\cite{Watson}), defined by the series
	\begin{equation}\label{2def: series expansion of J}
		J_{\varnu} (z) = \sum_{n=0}^\infty \frac {(-)^n \lp  z/2 \rp^{\varnu+2n } } {n! \Gamma (\varnu + n + 1) }.
	\end{equation}
	Moreover, we may write (see  \cite[\S 16.12, 16.3, 17.5]{Whittaker-Watson} or \cite[\S 7.2]{Watson})
	\begin{equation} \label{2eq: Bessel function and Whittaker function}
		J_{\varnu}(x) = \frac {1} {\sqrt {2 \pi x}} \left( e^{i x} W_{\varnu,\shskip +}(x) + e^{- i x} W_{ \varnu, \shskip -} (x) \right),
	\end{equation}
	with 
	\begin{equation}
		\label{2eq: bounds for Whittaker functions}
		x^j W_{\varnu, \shskip \pm}^{(j)} (x) \Lt_{\, \varnu, j} 1, \hskip 15pt x \Gt 1.
	\end{equation}
	
	\subsection{Asymptotic of a Bessel integral} 
	
	For a fixed (non-negative valued) bump function $U \in C_c^{\infty} (0, \infty)$, say with support in $[1,2]$, $a, b > 0$ and $X > 1$, consider the Bessel integral
	\begin{align}\label{3eq: defn I(a,b;X)}
		I_{k} (a, b; X) =  \int_0^\infty U\left( {x}/{X}\right) e ( {2  a \sqrt{  x}}  ) J_{k-1}  ( {4\pi b \sqrt{ x}}  ) \mathrm{d}x.
	\end{align}
	
	By \cite[6.699 1, 2]{G-R}, we have
	\begin{equation*}
		\begin{split}
			\int_0^\infty  e^{ i a x} J_{\varnu}(b x) x^{\shskip\mu-1} \mathrm{d}x =
			\frac{e^{\pi i (\varnu + \mu) /2 } b^\varnu}{2^{\varnu} a^{\varnu+\mu}} & \frac{\Gamma(\varnu+\mu)}{\Gamma(\varnu+1)}   F\left(\frac{\varnu+\mu}{2},\frac{\varnu+\mu+1}{2};\varnu+1; \frac{b^2}{a^2}\right)   
		\end{split}
	\end{equation*}
	for $b > a > 0$ and $ - \Re\, \varnu < \Re\, \mu <   3 / 2 $. By appealing to the Gaussian formula (see  \cite[\S 2.1]{MO-Formulas})
	\begin{equation*}
		\begin{split}
			F(\valpha,\beta;\gamma;1)=\frac{\Gamma(\gamma)\Gamma(\gamma-\valpha-\beta)}{\Gamma(\gamma-\valpha)\Gamma(\gamma-\beta)},\quad  \Re (\valpha+\beta-\gamma) < 0, \, \gamma \neq 0, -1, -2,...,
		\end{split}
	\end{equation*}
	and the duplication formula for the gamma function, we obtain
	\begin{equation}\label{3eq: J (ax) exp(iax)}
		\begin{split}
			\int_0^\infty  e^{ i a x} J_\varnu(a x) x^{\shskip\mu-1} \mathrm{d}x =
			\frac{  e^{\pi i (\varnu + \mu) /2 }   }{\sqrt \pi (2a)^{  \mu}} \frac {\Gamma (\varnu+\mu) \Gamma  (  1 / 2 - \mu  ) } {\Gamma (\varnu - \mu + 1) } , \quad  - \Re \,\varnu < \Re \,\mu < \frac 1 2 ,
		\end{split}
	\end{equation}
	after letting  $b \ra a$. Note that the limit  $b \ra a$ is legitimate because both the integral on the left and the hypergeometric series on the right are  absolutely and uniformly convergent for $ - \Re \,\varnu < \Re \,\mu <   1 / 2 $ (see  \cite[\S 2.1]{MO-Formulas}). 
	
	We first consider $ I_k(a, a; X)$ as defined in \eqref{3eq: defn I(a,b;X)}. By Mellin inversion
	\begin{equation*}
		\begin{split}
			I_k(a, a; X) = & \frac{X}{2\pi i}\int_{(\sigma)}\widetilde{U}(s) \int_0^\infty 2 \shskip e (2 a \hskip -1pt \sqrt { X } x ) J_{k-1} ( {4\pi a \hskip -1pt \sqrt{ X } x }  ) x^{1-2s} \mathrm{d}x \shskip \mathrm{d}s,
		\end{split}
	\end{equation*}
	where  $\widetilde{U}(s)$ denotes the Mellin transform of the function $U$, and $(\sigma)$ stands for the contour $\Re \, s = \sigma$ as usual. Applying \eqref{3eq: J (ax) exp(iax)} to evaluate the inner integral, we infer that 
	\begin{align*}
		I_k (a,a;X) = \frac{X}{2\pi i}\int_{(\sigma)}\widetilde{U}(s)\frac{2 i^{k-1}}{\sqrt{\pi} ( {-8\pi i a \hskip -1pt \sqrt{  X}}  )^{2-2s}} \frac{ \Gamma(k-2s+1)\Gamma(2s-3/2)} {\Gamma(k+2s-2)}\mathrm{d}s,
	\end{align*}
	for  $ {3}/{4}<  \sigma  < ( k+1)/2 $. Assume that $a^2 X > 1$. By  shifting the contour of integration to $\Re\, s = 0$, say, and collecting the residues at $s=3/4$ and $1/4$, we obtain the following asymptotic for $ I_k (a,a;X) $.
	
	\begin{lem}\label{pre-key-lemma} 
		We have
		\begin{equation}\label{3eq: I(a,a;X) asymptotic}
			\begin{split}
				I_k (a,a;X) = \frac {(1+i)i^{k-1} \widetilde{U}(3/4) X } {4 \pi  ( a^2 X)^{1/4} } + O \left( \frac {X} {(a^2  X)^{3/4 } } \right),
			\end{split}
		\end{equation}
		with the implied constant depending only on $k$ and $U$. 
	\end{lem}
	
	We now consider $ I_k (a, b; X) $ as in \eqref{3eq: defn I(a,b;X)} for $a \neq b$. For this, we assume that $ b^2 X > 1$ so that $J_{k-1}  ( {4\pi b \sqrt{x}}  )$ is oscillatory. In view of \eqref{2eq: Bessel function and Whittaker function} and \eqref{2eq: bounds for Whittaker functions}, the lemma below is a direct consequence of Lemma \ref{lem: staionary phase, dim 1, 2}.
	
	\begin{lem}\label{pre-key-lemma, 2} 
		Suppose that $b^2 X > 1$. Then $  I_k (a, b; X) = O (X^{-A})$ for any $A \geqslant 0$ if $ |   a -  b | \sqrt X > X^{\sepsilon}  $. 
	\end{lem}
	
	\subsection{Remarks on the Bessel integral}\label{sec:Weber-Hankel}
	After suitable changes, Weber's second exponential integral formula in \cite[13.31 (1)]{Watson} may be written as 
	\begin{equation}\label{2eq: Bessel product}
		\begin{split}
			\int_0^{\infty} \exp   (   - 2 \pi  x /X  )   J_{k-1} ( 4 \pi a \hskip -1pt \sqrt x) & J_{k-1} ( 4 \pi b \hskip -1pt  \sqrt x)  \shskip \nd x   \\
			& =  (X/2\pi) I_{k-1}  \lp  4 \pi {ab X}   \rp   \exp   \lp -   {2 \pi \big( a^2 + b^2 \big) X}  \rp    ,
		\end{split}
	\end{equation}
	for $a, b, X > 0$. Since $ J_{k-1} (4 \pi a \hskip -1pt \sqrt x) $ and $e ( 2 a \sqrt x)$ have the same type of oscillation (see \eqref{2eq: Bessel function and Whittaker function} or \cite[7.21 (1)]{Watson}), the Weber integral in \eqref{2eq: Bessel product} may be viewed as a variant of the Bessel integral in \eqref{3eq: defn I(a,b;X)}. However, the exponential function $  \exp   (   - 2 \pi  x /X  )$ is not as nice as the compactly supported function $U (x/X)$ from the perspective of Fourier analysis---the Fourier transform of $\exp   (   - 2 \pi  x /X  )$ ($x \in (0, \infty)$)  decays at $\infty$ only to the first order. 
	
	The connection between the Weber integral and the Dirac $\delta$-distribution might be of its own interest. This justifies the use of $\delta$ in another way. 
	
	According to \cite[7.23 (2)]{Watson}, we have the asymptotic $I_{k-1} (x) \sim \exp (x) / \sqrt {2\pi x}$ as $x \ra \infty$, so if one let $X \ra \infty$ then the right-hand side of \eqref{2eq: Bessel product} is asymptotic to 
	\begin{align*}
		\frac {\sqrt {2X} \exp  \lp - 2 \pi (a-b)^2 X \rp } {8 \pi^2 \hskip -1pt \sqrt {  ab  } }  = \frac {N (a-b, 1/ \hskip -1pt \sqrt {4\pi X })} { 8 \pi^2 \hskip -1pt \sqrt {ab} } \ra \frac {   \delta (a-b) } {8 \pi^2 b },
	\end{align*}
	where $ N(a-b, 1/ \hskip -1pt \sqrt {4\pi X}) $ is the Gaussian distribution of variance $1/ \hskip -1pt \sqrt {4\pi X}$ and $\delta (a-b)$ is the Dirac $\delta$-distribution. Thus the limiting form of \eqref{2eq: Bessel product} is 
	\begin{align}
		\int_0^{\infty}  J_{k-1} ( 4 \pi a \hskip -1pt \sqrt x)  J_{k-1} ( 4 \pi b \hskip -1pt  \sqrt x) \shskip \nd x = \frac {   \delta (a-b) } { 8 \pi^2 \hskip -1pt b },
	\end{align}
	or 
	\begin{align}
		\int_0^{\infty}  J_{k-1} ( a x)  J_{k-1} ( b x) x \shskip \nd x = \frac {   \delta (a-b) } {  b },
	\end{align}
	while this is equivalent to the Hankel inversion formula (see  \cite[14.3 (3), 14.4 (1)]{Watson})
	\begin{align}\label{2eq: Hankel inversion}
		\int_0^{\infty} x \shskip  \nd x \int_0^{\infty} F (a)  J_{k-1} (  a  x)  J_{k-1} (  b  x) a \shskip  \nd\shskip a = F (b) ,  
	\end{align}
	for $F  (a) \in C^{\infty} (0, \infty) $ subject to the condition
	\begin{align}
		\int_0^{\infty} |F (a)| \sqrt a \hskip 1pt \nd \shskip a < \infty .
	\end{align}

	\subsection{A Bessel $\delta$-method}
	
	By Lemma \ref{pre-key-lemma} and \ref{pre-key-lemma, 2}, we have the following asymptotic $\delta$-identity.
	
	\begin{lem}\label{lem: Bessel-delta}
		Let  $p$ be prime and $N, X > 1$ be such that  $ X > p^2 / N$ and $ X^{1-\sepsilon} >   N$. Let $r, n$ be integers in the dyadic interval $ [N, 2N]$. For any $A \geqslant 0$, we have
		\begin{equation}\label{2eq: Bessel-delta}
			\begin{split}
				\frac {2 \pi C_U r^{1/4} } { i^k p^{1/2} X^{3/4} } \cdot \frac 1 p \sum_{ a (\mod p) }  e \lp \frac {a (n-r) } p \rp  & \cdot I_k \lp \frac {\hskip -1pt \sqrt r} p,\frac {\hskip -1pt \sqrt n} p;X \rp  \\
				=  \delta({r = n})   &   \lp 1 + O_{k, \shskip U} \lp    \frac {p } {  \sqrt{NX} }  \rp \rp + O_{k, \shskip U, \shskip A} \big(X^{-A} \big), 
			\end{split}
		\end{equation} 
		where  $ C_U = (1+i)  / { \widetilde{U}(3/4)  } $, the $ \delta ({r = n}) $ is the Kronecker $\delta$ that detects $r = n$, and the implied constants depend only on $k$, $U$ and $A$. 
	\end{lem}
	
	\begin{proof}
		Lemma \ref{pre-key-lemma} yields the $\delta$-term, while Lemma \ref{pre-key-lemma, 2} implies that $I_k ( \hskip -1.5pt {\sqrt r}/ p, \hskip -1.5pt {\sqrt n}/ p ; X )$ is negligibly small unless $|r-n| \leqslant X^{ \sepsilon} p \sqrt {N/X}$. On the other hand, the exponential  sum in \eqref{2eq: Bessel-delta} gives us $r \equiv n (\mod p)$. Consequently, \eqref{2eq: Bessel-delta} follows immediately for $ X^{ \sepsilon} p \sqrt {N/X} < p $ as   assumed.
	\end{proof}

	\begin{rem}
		We should point out that the identity \begin{equation*}
			\begin{split}
				\frac{1}{p}\sum_{a (\mod p)}e\left(\frac{a(n-r)}{p}\right)= \delta({n\equiv r (\mod p)})
		\end{split}\end{equation*} 
		plays a key role in the work \cite{AHLS}. In fact, the approach therein is based on the observation:
		\begin{equation*}
			\begin{split}\sum_{r\shskip \sim N}\vchi (r)\sum_{n\sim X}\lambdaup_g(n)S(r,n;c)\approx X\sum_{n\sim N}\lambdaup_g(n)\vchi(n),
		\end{split}\end{equation*}
		where the modulus $c$ is chosen to be $c=p q\Gt N^{1+\sepsilon}$ and $X={p^2q^2}/{N}${\rm;} $\vchi$ is a primitive Dirichlet character modulo $q$. 
		
		Here the Bessel-exponential integral $I_k ( \hskip -1.5pt {\sqrt r}/ p, \hskip -1.5pt {\sqrt n}/ p ; X )$ serves the role of ``lowering" the conductor of the underlying problem.
	\end{rem}

	\section{\texorpdfstring{Application of the Bessel $\delta$-method and the \Voronoi summation}{Application of the Bessel $\delta$-method and the Voronoi summation}}
	
	We start with separating oscillations by writing
	\begin{align*}
		S(N) =   \sum_{n=1}^\infty \lambdaup_g(n)e(f (n))V\left(\frac{n}{N}\right)= \sum_{r=1}^\infty e(f (r))  V\left(\frac{r}{N}\right) \sum_{n=1}^\infty \lambdaup_g(n) \delta (r=n).
	\end{align*}
	Applying the $\delta$-method identity \eqref{2eq: Bessel-delta} in Lemma \ref{lem: Bessel-delta} and dividing the $a$-sum according as $(a, p) = 1$ or not, we have
	\begin{align*}
		S(N) =   {S}_p^{\star} (N, X) +   {S}_p^{ \scalebox{0.65} 0} (N, X) +  R_p(N, X) +  O \big(X^{-A} \big),
	\end{align*}
	with
	\begin{equation}
		\begin{split}
			{S}_p^{\star} (N, X) =  \frac {2 \pi i^{k} {M^{1/2} N^{1/4}  } } {   \eta_g   p^{3/2} X^{3/4} } & \sum_{r=1}^\infty e(f (r))  V_{\scriptscriptstyle \natural}   \hskip -1pt \left(\frac{r}{N}\right)   \sumx_{ a (\mod p) }  e \lp - \frac {a r } p \rp \\
			\cdot	&\sum_{n=1}^\infty \lambdaup_g(n) e \lp \frac {a n } p \rp I_k \lp \frac {\hskip -1pt \sqrt r} p,\frac {\hskip -1pt \sqrt n} p;X \rp,
		\end{split} 
	\end{equation}
	\begin{equation}
		\begin{split}
			{S}_p^{ \scalebox{0.65} 0} (N, X)  =  \frac {2 \pi i^{k} {M^{1/2} N^{1/4}  } } {   \eta_g  p^{3/2} X^{3/4} } \sum_{r=1}^\infty e(f (r))  V_{\scriptscriptstyle \natural}     \hskip -1pt \left(\frac{r}{N}\right)  \hskip -2pt \sum_{n=1}^\infty \lambdaup_g(n) I_k \lp \frac {\hskip -1pt \sqrt r} p,\frac {\hskip -1pt \sqrt n} p;X \rp,
		\end{split} 
	\end{equation}
	where $V_{\scriptscriptstyle \natural} (x) =  C_U \eta_g \xi (-1)  M^{-1/2} \cdot x^{1/4} V (x) $ (recall that $\xi (-1) = (-1)^k$) and {\small $\displaystyle \sumx$} means that the $a$-sum is subject to $(a, p) = 1$, and
	\begin{equation}
		R_p (N, X) = O \bigg( \frac {p } {  \sqrt{NX} } \sum_{n \sim N} |\lambdaup _g(n)|  \bigg)  = O \bigg(  p \sqrt {\frac N X}   \bigg) .
	\end{equation}
	Assuming $p > M$, we now apply the \Voronoi summation in Lemma \ref{lem: Voronoi} to  the $n$-variable. Recall from \eqref{3eq: defn I(a,b;X)} that
	\begin{align*}
		I_k \lp \frac {\hskip -1pt \sqrt r} p,\frac {\hskip -1pt \sqrt n} p;X \rp =  \int_0^\infty U\left( {x}/{X}\right) e \lp \frac {2   \sqrt{ r x}} {p} \rp J_{k-1} \lp \frac {4\pi  \sqrt{ n x}} p  \rp  \mathrm{d}x, 
	\end{align*}
	and the integral may be regarded  as a Hankel transform as in \eqref{2eq: Hankel transform}.	By applying the (complex conjugation of) \Voronoi summation in \eqref{1eq: Voronoi} with $c = p$ in the {\it reversed} direction,  we infer that 
	\begin{align}\label{4eq: Sx after Voronoi}
		{S}_p^{\star} (N, X)  \hskip -0.5pt =  \hskip -0.5pt  { \frac { \xi (p)  N^{1/4} } { p^{1/2} X^{3/4} } } & \sum_{r=1}^\infty e(f (r))  V_{\scriptscriptstyle \natural}    \hskip -1.5pt \left(\frac{r}{N}\right)    \hskip -1pt
		\sum_{n=1}^\infty \overline {\lambdaup_{  g } (n)} S(n, r; p)  e \hskip -1pt \lp \frac {2 \sqrt {n r}} { \sqrt M p } \rp \hskip -1pt U \hskip -1pt \left(\frac {n} {M X}\right) \hskip -1pt ,
	\end{align}
	where, as usual, $S (n, r; p)$ is the Kloosterman sum 
	\begin{align*}
		S (n, r; p) = \sumx_{ a (\mod p) }  e \lp  \frac {a n + \overline{a} r } p \rp.
	\end{align*}
	Similarly, 
	\begin{align}\label{4eq: S0 after Voronoi}
		{S}_p^{\scalebox{0.65} 0} (N, X)   =  \frac { p^{1/2} N^{1/4} } {  X^{3/4} } \sum_{r=1}^\infty e(f (r))  V_{\scriptscriptstyle \natural}    \hskip -1.5pt \left(\frac{r}{N}\right)    \hskip -1pt
		\sum_{n=1}^\infty \overline {\lambdaup_{  g } (n)} e \lp \frac {2 \sqrt {n r}} { \sqrt M  } \rp U\left(\frac {p^2 n} {M X}\right),
	\end{align}
	after the Vorono\"i with modulus $c = 1$. Estimating trivially, we find that
	\begin{align}\label{4eq: bound for S0}
		{S}_p^{\scalebox{0.65} 0} (N, X) \Lt  \frac { N^{5/4} X^{1/4} } {p^{3/2} }. 
	\end{align}
	Finally, we introduce an average over primes $p$ in $[P, 2P]$ for a large parameter $P$; there are $\asymp P/\log P$ many such $p$'s. The results that we have established are summarized as follows.
	
	\begin{prop}\label{key lemma}
		Let $V (x) \in C_c^{\infty} (0, \infty) $ be supported in $[1, 2]$, with $\mathrm{Var} (V) \Lt 1$ and $V^{(j)} (x) \Lt_{j} \varDelta^j$  for $j \geqslant 0$.	Let parameters $N, X, P > N^{\shskip \sepsilon}$ be such that
		\begin{align}\label{4eq: asummptions on N,X,P}
			P^2/N < X, \hskip 10pt N < X^{1-\sepsilon}.
		\end{align} 
		Let $P^{\star}$ be the number of primes in $[P, 2P]$.
		We have
		\begin{equation}\label{4eq: S(N)=S(N,X,P)}
			\begin{split}
				S (N) = \sum_{n=1}^\infty \lambdaup _g(n)e(f (n)) V\left(\frac{n}{N}\right) = {S}(N, X, P)
				+ O  \left( \frac {P \hskip -1pt \sqrt N } { \sqrt X } + \frac{N^{5/4}X^{1/4}}{P^{3/2}}\right),
			\end{split}
		\end{equation}
		with
		\begin{equation}\label{beginning object}
			\begin{split}
				{S}(N, X, P) = \frac{ N^{1/4} }{P^{\star} X^{3/4} }  \sum_{ p \shskip \sim P}   
				\frac {\xi (p)}   {\sqrt {p}}
				& \sum_{r=1}^{\infty}e(f (r))  V_{\scriptscriptstyle \natural}   \left(\frac{r}{N}\right) \\
				\cdot &\sum_{n=1}^{\infty} \overline {\lambdaup_{  g } (n)} S(n,r;p)
				e \hskip -1pt \lp \frac {2 \sqrt {n r}} { \sqrt M p } \rp \hskip -1pt U \hskip -1pt \lp  \frac { n } { M X }  \rp \hskip -1pt,
			\end{split}
		\end{equation}
		where $V_{\scriptscriptstyle \natural}    (x) =  C_U \eta_g \xi (-1)  M^{-1/2} \cdot x^{1/4} V (x) $  is again supported in $[1, 2]$, satisfying $\mathrm{Var}({V_{\scriptscriptstyle \natural}}) \Lt 1$ and $V_{\scriptscriptstyle \natural}^{(j)} (x) \Lt_{j} \varDelta^j$. 
	\end{prop}

	\section{Application of the Poisson summation and the Cauchy inequality}
	
	In view of Proposition \ref{key lemma}, to study $S (N)$ it suffices to consider the sum ${S}(N, X, P)$ defined in \eqref{beginning object}. 
	For convenience of our analysis, we let 
	\begin{align}\label{5eq: XN=P2K2}
		X = {P^2K^2}/{N},  \hskip 15pt N^{\shskip\sepsilon} < K < T^{1-\sepsilon},
	\end{align}
	with the parameter $K$ to be optimized later. Then the first assumption in \eqref{4eq: asummptions on N,X,P} is justified, while the second assumption $N<X^{1-\sepsilon}$  amounts to
	\begin{align}
		\label{5eq: assumption on K} 
		P> {N^{1+\sepsilon}} / {K}.
	\end{align}  
	
	\subsection{First application of the Poisson summation} Recall  that $f (r) = T \phi (r/N) + \gamma r$ (as in  \eqref{0eq: f = T phi}). By applying the Poisson summation to the $r$-sum in \eqref{beginning object}, we have
	\begin{align*}
		\sum_{r=1}^{\infty} e(f (r))  S(n,r;p)  e \hskip -1pt \lp \frac {2 \sqrt {n r}} { \sqrt M p } \rp  \hskip -1pt V_{\scriptscriptstyle \natural}    \hskip -1pt \left(\frac{r}{N}\right) = N \sum_{(r, \shskip p) = 1}  e \hskip -1pt\left( - \frac{\overline{r} n }{p}\right)\CalJ (n,r,p),
	\end{align*}
	where 
	\begin{equation}\label{integral}
		\begin{split}
			\CalJ (y,r,p)=\int_{0 }^{\infty} V_{\scriptscriptstyle \natural}   (x) e \hskip -1pt  \left( T \phi (x) + \gamma N x + \frac{2\sqrt{Nx y}}{ \sqrt M p} -\frac{rNx}{p}\right)   \mathrm{d}x.
		\end{split}
	\end{equation}
	Recall that $y$ ($=n$) $\sim   M X$. Thus the derivative of the phase function is
	\begin{align*}
		N (\gamma - r/p) + T \phi'(x) + \frac {\sqrt{N y}} {\sqrt{M x} p} = N (\gamma - r/p) + O \bigg( T + \frac {\sqrt{N X}} {P} \bigg),
	\end{align*}
	and hence it is dominated by $ N (\gamma - r/p) $  when $$  N |r / p - \gamma |  \Gt  \max \big\{ T ,   \sqrt {NX}  / {P} \big\} =   \max \left\{ T,   K  \right\} = T . $$ Under this condition, Lemma \ref{lem: staionary phase, dim 1, 2}   implies that $\CalJ (y,r,p)$ is negligibly small, provided that $\phi^{(j)} (x) \Lt_j 1$ ($j \geqslant 1$) and that $V_{\scriptscriptstyle \natural}^{(j)} (x) \Lt_{j} \varDelta^j$ for $\varDelta \leqslant T / N^{\sepsilon} $.  
	Accordingly, 
	set 
	\begin{align}\label{5eq: R = Pt/N}
		R =    P T / N  .
	\end{align}
	So if we  assume that
	\begin{align}\label{5eq: Delta < T1}
		\varDelta \leqslant T /N^{\sepsilon}, 
	\end{align} 
	then we can effectively truncate the sum at $|r - \gamma p| \sasymp R$, at the cost of a negligible error. Note that \eqref{5eq: Delta < T1} amounts to the condition \eqref{1eq: condition on Delta} in Theorem \ref{main-theorem}.
	
	Moreover, the second derivative test in Lemma \ref{lem: derivative tests, dim 1}  immediately yields the following estimate for $ \CalJ (y,r,p) $.
	\begin{lem}\label{lem: bound for J}
		Suppose that $|\phi'' (x) | \Gt 1$. Then, for $1 \leqslant y / MX \leqslant 2 $, we have
		\begin{align}\label{5eq: bound for Jt(npr),2}
			\CalJ (y,r,p) \Lt  \frac  1 {\sqrt {T } } .
		\end{align}
	\end{lem}
	\begin{proof}
		The second derivative of the phase function in \eqref{integral} is equal to 
		\begin{align*}
			T \phi''(x) - \frac {\sqrt{N y} } {2 p \sqrt{M x} x} = T \phi''(x) + O (K). 
		\end{align*}
		By our assumptions, $|\phi'' (x) | \Gt 1$ and $K < T^{1-\sepsilon}$, the estimate above follows easily from Lemma \ref{lem: derivative tests, dim 1}.
	\end{proof}
	Consequently,  \eqref{beginning object} is transformed into
	\begin{equation}
		\begin{split}
			{S}(N, X, P) \hskip -1pt =  \hskip -1pt \frac{N^{2 }}{P^{\star} (P K)^{3/2} } \hskip -1pt \sum_{n=1}^{\infty} \hskip -1pt \overline {\lambdaup_{  g } (n)} 
			U \hskip -1pt \left(\frac{n}{MX}\right) \hskip -1.5pt
			\sum_{p \shskip \sim P} \hskip -1.5pt \frac {\xi (p)} {\sqrt p} \hskip -1.5pt
			\sum_{\sstyle (r,\, p)=1 \atop {\sstyle |r-\gamma p| \Lt \shskip R}} \hskip -1.5pt e\left(\hskip -1pt -\frac{\overline{r}n }{p} \hskip -1pt \right)\CalJ (n,r,p) & \\
			+ O \big(N^{-A}\big) &.
		\end{split}
	\end{equation}

	\subsection{Application of the Cauchy inequality and the second Poisson summation} Next we apply Cauchy and the Ramanujan bound on average for  the Fourier coefficients $  {\lambdaup_{  g } (n)} $ as in \eqref{2eq: Ramanujan}. Thus,
	\begin{equation*}
		\begin{split}
			{S}(N, \hskip -0.5pt X, \hskip -0.5pt P) \hskip -1pt \Lt_{g} \hskip -1pt \frac{N^{3/2 }}{P^{\star} \hskip -1pt \sqrt{PK} }
			\Bigg(\sum_{n=1}^{\infty}\Bigg|\hskip -1pt \sum_{p \shskip \sim P} \hskip -1pt \frac {\xi (p)} {\sqrt p} \hskip -2pt
			\sum_{\sstyle (r, \shskip p)=1 \atop {\sstyle |r-\gamma p| \Lt R}}\hskip -2pt 
			e\left( \hskip -1pt  - \frac{\overline{r}n }{p} \hskip -1pt \right) \hskip -1pt
			\CalJ (n,\hskip -0.5pt r,\hskip -0.5pt p) \Bigg|^2 \hskip -1pt U \left(\frac{n}{M X}\right) \hskip -1pt \Bigg)^{\hskip -2pt 1/2}.
	\end{split}\end{equation*}
	Opening the square and switching the order of summations, the square of the right-hand side is
	\begin{equation}\label{afterCauchy-Schwarz}
		\begin{split}
			\scalebox{0.98}{\text{$\displaystyle\frac{N^{3 }}{{P^{\star 2}}  P K} \hskip -1.5pt \underset{p_1,\shskip p_2 \shskip \sim P}{\sum \sum}  \frac { \xi (p_1 \overline p_2)  } {\sqrt {p_1 p_2} } \hskip -1.5pt \underset{\sstyle (r_i ,\shskip p_i)=1  \atop {\sstyle |r_i - \gamma p_i|  \Lt R }}{\sum \sum}  \sum_{n=1}^{\infty} e \hskip -1.5pt \left(\frac{\overline{r}_2 n }{p_2} \hskip -1pt - \hskip -1pt \frac{\overline{r}_1 n }{p_1} \hskip -1pt \right) \hskip -1pt \CalJ (n,r_1,p_1)\overline{\CalJ (n,r_2,p_2)} U \hskip -1pt\left(\hskip -0.5pt \frac{n}{M X} \hskip -0.5pt \right) \hskip -1.5pt .$}}	
		\end{split}
	\end{equation}
	
	\begin{rem}\label{remark diagonal}
		To keep in mind some representative cases, we notice that the diagonal contribution $(p_1,r_1)=(p_2,r_2)$ towards ${S}(N, X, P)$ is
		\begin{equation*}
			\begin{split}
				\frac{N^{3/2 }}{P^{\star} \hskip -1pt \sqrt {PK} } \Bigg(\underset{p  \shskip \sim P}{\sum } \frac { 1 } {p  }  \hskip -1pt \underset{\sstyle (r,\, p)=1 \atop {\sstyle |r - \gamma p | \shskip  \Lt   R}}{\sum }  \sum_{n=1}^{\infty}   \big|\CalJ (n,r,p)\big|^2 U \hskip -1.5pt\left(\hskip -1pt \frac{n}{MX} \hskip -1pt \right) \Bigg)^{\hskip -2pt 1/2}   \Lt_{ M} \sqrt {NK \log P}  ,
		\end{split}\end{equation*}
		where $R = P T / N$   as in {\rm\eqref{5eq: R = Pt/N}} and we have used the bound in Lemma {\rm\ref{lem: bound for J}} for $\CalJ (n,r,p) $. 
		
		It is therefore important to introduce the extra average over $p$ as in {\rm\eqref{beginning object}}, because without  it the diagonal contribution would be $O(\hskip -1.5pt\sqrt {p N K} )$ instead. 
	\end{rem}
	
	We then apply Poisson summation with modulus $p_1p_2$ (note that $p_1$ and $p_2$ need not be distinct) to the $n$-sum in \eqref{afterCauchy-Schwarz}, getting
	\begin{equation}\label{n sum after poisson}
		\begin{split}	& \hskip 13pt \sum_{n=1}^{\infty} e \hskip -1pt \left(\frac{\overline{r}_2 n }{p_2}  - \frac{\overline{r}_1 n }{p_1} \hskip -1pt \right) \hskip -1pt
			\CalJ (n,r_1,p_1)\overline{\CalJ (n,r_2,p_2)} 
			U \hskip -1pt\left(\hskip -1pt \frac{n}{MX} \hskip -1pt \right)\\
			&=\frac{MX}{p_1p_2}\sum_{n=-\infty}^{\infty} \, \sum_{a (\mod p_1p_2)}e\left( \frac{a\overline{r}_2 }{p_2} -\frac{a\overline{r}_1  }{p_1} + \frac{an}{p_1p_2}\right) \hskip -1pt \cdot \hskip -1pt \CalL\left(\frac{MXn}{p_1p_2}; r_1,r_2,p_1,p_2\right),
		\end{split}
	\end{equation}
	with 
	\begin{align}\label{5eq: I integral}
		\CalL ( x  ) = \CalL ( x; r_1,r_2,p_1,p_2  )  = \int_{0 }^{\infty} U(y)  \CalJ (MXy,r_1,p_1) \overline{\CalJ (MXy,r_2,p_2)} \,e\left(- x y\right)\mathrm{d}y.
	\end{align}
	Recall that $\sqrt {NX} = P K$ as in  \eqref{5eq: XN=P2K2} and that $  \CalJ (MXy,r,p) $ is defined as in \eqref{integral}. We have
	\begin{equation}\label{5eq: I = triple integral}
		\begin{split}
			\CalL  ( x  ) \hskip -1pt =   \hskip -1pt \int_{0 }^{\infty} \hskip -2pt \int_{0 }^{\infty} \hskip -1pt V_{\scriptscriptstyle \natural} (\varv_1) \overline{V_{\scriptscriptstyle \natural}  (\varv_2)}  e \hskip -1pt \left( \hskip -1pt T (\phi (\varv_1) \hskip -1pt - \hskip -1pt \phi (\varv_2)) \hskip -1pt + \hskip -1pt \gamma N ( \varv_1\hskip -1pt - \hskip -1pt  \varv_2) \hskip -1pt - \hskip -1pt \frac{Nr_1\varv_1}{p_1} 
			\hskip -1pt + \hskip -1pt \frac{Nr_2\varv_2}{p_2} \hskip -1pt \right)  & \\
			\cdot	\int_{0 }^{\infty} U (y)  e \hskip -1pt \left(2   PK  \lp \frac{ \sqrt{ \varv_1 }}{p_1}  - \frac{ \sqrt{ \varv_2 }}{p_2} \rp \hskip -1pt \sqrt { y} - x y \right) \mathrm{d}y\, \mathrm{d} \varv_2 \, \mathrm{d} \varv_1 . &
		\end{split}
	\end{equation}
	We note that the $a$-sum in \eqref{n sum after poisson} yields the congruence condition 
	\begin{equation}\label{5eq: congruence condition}
		\begin{split} n \equiv \overline{r}_1 p_2- \overline{r}_2 p_1     \, (\mod p_1p_2),
		\end{split} 
	\end{equation} 
	where $\overline{r}_1$ and $\overline{r}_2$ denote the multiplicative inverses of $r_1$ and $r_2$ modulo $p_1$ and $p_2$ respectively. Thus the right-hand side of \eqref{n sum after poisson} is simplified to 
	\begin{align}\label{5eq: n sum after Poisson, 2}
		{MX} \sum_{ n \shskip \equiv \shskip \overline{r}_1 p_2- \overline{r}_2 p_1      (\mod p_1p_2) }   \CalL\left(\frac{MXn}{p_1p_2}; r_1,r_2,p_1,p_2\right). 
	\end{align}
	
	\subsection{Analysis of the integral $\CalL (x)$} This section is dedicated to the analysis of the integral  $\CalL (x)$ as defined in \eqref{5eq: I integral} or \eqref{5eq: I = triple integral}. 
	
	By applying Lemma \ref{lem: bound for J} to the integral in \eqref{5eq: I integral}, we obtain the following (trivial) estimate.
	
	\begin{lem} \label{lem: bound for I, 0}
		We have
		\begin{align}\label{5eq: bound for I, 02}
			\CalL (x) \Lt \frac  1    {T } .
		\end{align}
	\end{lem}
	
	Further, we wish to improve the estimate above by examining the triple integral in \eqref{5eq: I = triple integral}.
	
	We first investigate  in the lemma below the $y$-integral in the second line of {\rm\eqref{5eq: I = triple integral}}; it is the integral $\CalK  (\varw K, x  )$ defined as in \eqref{5eq: the integral K}, with $\varw = \sqrt{ \varv_1 } P /{p_1}  -  \hskip -1.5pt \sqrt{ \varv_2 } P/ {p_2}$. 
	
	\begin{lem}\label{lem: the integral K}
		Let $K > N^{\sepsilon} > 1$.		For real $\varw, x$, with $|\varw| \leqslant \hskip -1pt \sqrt 2 - 1/2 < 1$,  define 
		\begin{align}\label{5eq: the integral K}
			\CalK  (\varw K, x ) = \int_{0 }^{\infty} U (y)  e \hskip -1pt \left(2  \varw K \hskip -1pt \sqrt { y} - x y \right) \mathrm{d}y .
		\end{align} 
		
		{\rm(1).} We have $\CalK  (\varw K, x ) = O (N^{-A})$ if $ |x| \geqslant K $. 
		
		{\rm(2).}  For $ |x| > N^{\sepsilon}$, we have $\CalK  ( \varw K, x ) = O (N^{-A})$ unless $   2/3 < \varw K / x < \hskip -1pt 3/2 $, say, and for $  1/  2 < \varw K / x < 2  $, if we let $\lambdaup = K^2 \varw^2 / x$ and $ W(\lambdaup) = W (\lambdaup, x) = e (- \lambdaup) \CalK  \big( \hskip -1.5pt \sqrt{\lambdaup x}, x  \big) $ then
		\begin{align}\label{5eq: bounds for U}
			\lambdaup^{j} W^{(j)} (\lambdaup) 
			\Lt_{\, j}   1 / {\hskip -1pt {\textstyle \sqrt {|x|}} }. 
		\end{align}

		{\rm(3).} $ \CalK  (\varw K, 0 ) = W_0 (2 \varw K ) $ for some Schwartz function $W_0 $.
		
	\end{lem}

	\begin{proof}
		The statements in (1) and the first part of (2) follow from Lemma \ref{lem: staionary phase, dim 1, 2}. (3) is also clear, for $ W_0 $ is the inverse Fourier transform of the function given by $2 y U (y^2)$ for $y > 0$ and by $0$ for $y \leqslant 0$. It is left to prove \eqref{5eq: bounds for U} for $ 1/4 < \lambdaup   / x < 4 $. For this, we change the variable $y$ to $\lambdaup y / x =  \varw^2 K^2 y / x^2 $ in \eqref{5eq: the integral K} so that 
		\begin{align*}
			W (\lambdaup, x) = \frac {\lambdaup} {x} \int_0^{\infty} U (\lambdaup y/x) e \lp - \lambdaup  ( 1 - 2 \sqrt y + y  ) \rp \nd y. 
		\end{align*}
		Then the estimates in \eqref{5eq: bounds for U} follow from Lemma \ref{lem: stationary phase estimates, dim 1}. 
	\end{proof}
	
	\begin{lem}\label{lem: the integral I}
		Let $N, T, K, P > 1$ be parameters with $N^{\sepsilon} < K \Lt T$ and $ N^{1+\sepsilon} < P K$. Let $p_i \sim P$ and $|r_i - \gamma p_i| \Lt PT/N$ {\rm(}$i = 1, 2${\rm)}. Suppose that  $  \phi^{(j)} (\varv) \Lt  1 $ for $j = 2, 3$ and that $ |\phi'' (\varv)| \Gt 1 $ for all $\varv \in (1/2, 5/2)$. Let the integral  $\CalL (x)$ be  as in {\rm\eqref{5eq: I = triple integral}}. 
		
		{\rm (1).} We have $\CalL (x) = O (N^{-A})$ if $ |x| \geqslant K  $. 
		
		{\rm(2).} Assume that $K^2 / T > N^{\sepsilon}$. For  $ K^2 / T  \Lt |x| < K $, we have
		\begin{align}\label{5eq: bound for I(x)}
			\CalL (x) \Lt  \frac {1}  {\textstyle T \hskip -1pt \sqrt  {|x|}   } .
		\end{align}
		{For $ |x| \Lt K^2 / T $, we have 
			\begin{align}\label{5eq: bound for I(x), 2}
				\CalL (x) \Lt \frac  {1}   T.
		\end{align} }
		
		{\rm(3).} Let $p_1 = p_2 = p$.  Then
		\begin{align}\label{5eq: bound for I(0)}
			\CalL (0) \Lt   \min \left\{  \frac 1 T, \frac {P N^{\sepsilon} } {KN |r_1-r_2|} \right\}.
		\end{align}
	\end{lem}
	
	\begin{proof}
		The statement in (1) is obvious in view of Lemma \ref{lem: the integral K} (1). 
		
		We then turn to the proof of (2) in the first case when $ K^2 / T  \Lt |x| < K $.
		First of all, by Lemma \ref{lem: the integral K} (2), we may write the integral in  {\rm\eqref{5eq: I = triple integral}} as below,
		\begin{align*}
			\CalL (x) = \frac 1 {  \sqrt {|x|} } \iint V_{\scriptscriptstyle \natural} (\varv_1)  \overline{V_{\scriptscriptstyle \natural}  (\varv_2)} W_{\scriptscriptstyle \natural} (\varw K/x) e  ( f (\varv_1, \varv_2) )   \shskip \mathrm{d} \varv_1 \shskip \mathrm{d} \varv_2 + O \big(N^{-A} \big) ,
		\end{align*}
		where $\varw  = \hskip -1pt \sqrt{ \varv_1 } P /{p_1}  -  \hskip -1pt \sqrt{ \varv_2 } P/ {p_2}$,  $W_{\scriptscriptstyle \natural} (y) = \sqrt {|x|} W (x y^2) F (y)$ for $W$ defined as in Lemma \ref{lem: the integral K} (2) and $F $ a smooth function  supported in $[1/2, 2]$, with $F \equiv 1$ on $\big[2/3, 3/2 \big]$, and  
		\begin{align*}
			f ( \varv_1, \varv_2 ) \hskip -1pt = \hskip -1pt T \big( \phi (\varv_1)  - \phi (\varv_2) \big) \hskip -1pt & +   \gamma N ( \varv_1\hskip -1pt - \hskip -1pt  \varv_2) \hskip -1pt - \hskip -1pt N \left(  \frac{ r_1\varv_1}{p_1} \hskip -1pt - \hskip -1pt \frac{ r_2\varv_2}{p_2}\right) \hskip -1pt \\
			& + \hskip -1pt \frac {K^2 P^2 } {x } \bigg(  \frac {\varv_1} {p_1^2} \hskip -1pt + \hskip -1pt \frac {\varv_2} {p_2^2} \bigg) \hskip -1pt - \hskip -1pt \frac {2K^2 P^2 \hskip -1pt \sqrt { \varv_1 \varv_2} } {x p_1 p_2} .
		\end{align*}
		In view of \eqref{5eq: bounds for U}, we have
		\begin{align*}
			W_{\scriptscriptstyle \natural}^{(j)} (y) \Lt_{\, j} 1. 
		\end{align*} 
		By Fourier inversion, we write
		\begin{align*}
			W_{\scriptscriptstyle \natural} (y) = \int_{-\infty}^{\infty} \widehat{W}_{\scriptscriptstyle \natural} (\varv) e (\varv y ) \mathrm{d} \varv, 
		\end{align*}
		where $\widehat{W}_{\scriptscriptstyle \natural} $ is the Fourier transform of $ W_{\scriptscriptstyle \natural} $, satisfying
		\begin{align*}
			\widehat{W}_{\scriptscriptstyle \natural} (\varv) \Lt (1 +|\varv|)^{-A}.
		\end{align*} 
		Thus, we may further write 
		\begin{align*}
			\CalL (x) = \frac 1 {\sqrt {|x|} } \int_{-N^{\ssepsilon}}^{N^{\ssepsilon}}  
			\widehat{W}_{\scriptscriptstyle \natural} (\varv)  \iint V_{\scriptscriptstyle \natural} (\varv_1)  \overline{V_{\scriptscriptstyle \natural}  (\varv_2)} e  ( f (\varv_1, \varv_2; \varv) )   \shskip \mathrm{d} \varv_1 \shskip \mathrm{d} \varv_2 \shskip \mathrm{d} \varv + O \big(N^{-A} \big),
		\end{align*}
		with 
		\begin{align*}
			f (\varv_1, \varv_2; \varv) = f (\varv_1, \varv_2) + \frac {K P \varv} {x} \lp \frac {\sqrt{\varv_1}} {p_1 } - \frac {\sqrt{\varv_2}} {p_2 } \rp . 
		\end{align*}
		We have
		\begin{align*}
			\partial^2 f  (\varv_1, \varv_2; \varv)  /\partial \varv_1^2 =   T \phi'' (\varv_1)  + \frac { K^2 P^2 \hskip -1pt \sqrt{ \varv_2}} {2 x p_1 p_2 \hskip -1pt \sqrt{\varv_1} \varv_1  } - \frac {KP \varv } {4 x p_1 \hskip -1pt \sqrt{\varv_1} \varv_1 },
		\end{align*}
		\begin{align*}
			\partial^2 f  (\varv_1, \varv_2; \varv)  /\partial \varv_2^2 = -   T \phi'' (\varv_2)  + \frac { K^2 P^2 \hskip -1pt \sqrt{ \varv_1}} {2 x p_1 p_2 \hskip -1pt \sqrt{\varv_2} \varv_2  } + \frac {KP \varv } {4 x p_2 \hskip -1pt \sqrt{\varv_2} \varv_2 } ,
		\end{align*}
		\begin{align*}
			\partial^2 f  (\varv_1, \varv_2; \varv) / \partial \varv_1 \partial \varv_2 =  -   \frac {K^2 P^2  } {2 x p_1 p_2 \hskip -1pt \sqrt { \varv_1 \varv_2}}. 
		\end{align*} 
		Since $ \phi'' (\varv) \Gt 1 $, when $ K^2 / |x| \Lt T $, it is clear that 
		\begin{align*}
			\big|\partial^2 f   /\partial \varv_1^2 \big|, \, \big|\partial^2 f   /\partial \varv_2^2 \big|  \Gt  T, \hskip 10pt  \big|\partial^2 f   / \partial \varv_1 \partial \varv_2 \big| \Lt K^2 /|x|, \hskip 10pt | \det f'' | \Gt T^2 
		\end{align*}
		for $1 \leqslant \varv_1, \varv_2 \leqslant 2$ and  $|\varv| \leqslant N^{\sepsilon}$. 
		We obtain the estimate in \eqref{5eq: bound for I(x)} by applying the two-dimensional second derivative test in Lemma \ref{lem: 2nd derivative test, dim 2} with $\lambdaup = \rho = T $. 
		
		In the second case in (2) when $  |x| \Lt K^2 / T $ is small, the estimate in \eqref{5eq: bound for I(x), 2} is just \eqref{5eq: bound for I, 02} in Lemma \ref{lem: bound for I, 0}.
		
		
		
		Finally, let us consider (3). The bound $ \CalL (0) \Lt 1/ T $ is already contained in \eqref{5eq: bound for I(x), 2}. 
		Now assume that $|r_1 - r_2| > P T / K N^{1-\sepsilon}$. In view of Lemma  \ref{lem: the integral K} (3), we may write
		\begin{align*}
			\CalL (0) =  \int_{-N^{\ssepsilon} / K}^{N^{\ssepsilon} / K} W_0    (2\varw K) \int_1^2  V_0    (\varw, \varv_2)  e  ( f_0 (\varw, \varv_2)  )    \shskip \mathrm{d} \varv_2 \shskip \mathrm{d} \varw + O \big(N^{-A} \big) ,
		\end{align*}
		where $V_0    (\varw, \varv_2) = (2p/P) \big( p\varw / P + \hskip -1pt \sqrt {\varv_2} \big)  V_{\scriptscriptstyle \natural} \big(\big( p\varw / P + \hskip -1pt \sqrt {\varv_2} \big)^2 \big)  \overline{V_{\scriptscriptstyle \natural}  (\varv_2)} $, satisfying 
		\begin{align*}
			\mathrm{Var} (V_0 (\varw, \cdot) ) = \int_1^2 |\partial V_0 (\varw, \varv_2) / \partial \varv_2 | \shskip \mathrm{d} \varv_2 \Lt 1,
		\end{align*}
		and
		\begin{align*}
			f_0 (\varw, \varv_2) = T \big(   \phi \big( (   {p \varw} / {P} + \hskip -1pt \sqrt {\varv_2} )^{2} \big) - \phi (\varv_2)  \big) - \frac {N} {p }    (  r_1    -    r_2) \varv_2 & \\
			-    \frac {2 N (r_1 - \gamma p) } { P  }  \hskip -1pt \sqrt {\varv_2} \varw     
			-    \frac {N  p (r_1-\gamma p)} {P^2}   \varw^2  \hskip -1pt & .
		\end{align*}
		Recall that $\phi^{(j)} (\varv) \Lt  1$ ($j = 2, 3$). For $  | r_1 -\gamma p|   \Lt P T /N$ and $|\varw| < N^{\sepsilon} / K$, we have
		\begin{align*}
			\partial f_0 ( \varw, \varv_2 ) /\partial \varv_2 & = - \frac {N} {p} (r_1-r_2) - \frac { N (r_1 - \gamma p) } { P  }  \frac {\varw} {\hskip -1pt \sqrt {\varv_2}} \\
			& \quad \,  + \frac {    T } { { \sqrt {\varv_2}} } \big(  \phi' \big( (   {p \varw} / {P} + \hskip -1pt \sqrt {\varv_2} )^{2} \big)    - \sqrt {\varv_2} \phi' (\varv_2) \big) \\ 
			& = - \frac {N} {p} (r_1-r_2) + O \lp\frac {T N^{\sepsilon} } {K} \rp ,
		\end{align*}
		and, similarly,
		\begin{align*}
			\partial^2 f_0 ( \varw, \varv_2 ) /\partial \varv_2^2 = O (   {T N^{\sepsilon} } / {K} ). 
		\end{align*}
		It follows from $|r_1 - r_2| > P T / K N^{1-\sepsilon}$ that  $ |\partial f_0 ( \varw, \varv_2 ) /\partial \varv_2| \Gt N |r_2-r_1| / P  $. By partial integration (the first derivative test), we infer that $ \CalL (0) \Lt P / K N^{1-\sepsilon} |r_1-r_2| $, as desired.
	\end{proof}
	
	\subsection{Estimates for ${S}(N,   X,   P)$}   
	
	Combining \eqref{afterCauchy-Schwarz}, \eqref{n sum after poisson} and \eqref{5eq: n sum after Poisson, 2}, along with Lemma \ref{lem: the integral I}, we conclude that
	\begin{align}\label{5eq: S<Sdiag+Soff}
		{S}(N,   X,   P)  \Lt_{M} \textstyle \hskip -1pt \sqrt {  {S}_{\mathrm{diag}}^2 (N,   X,   P)} + \hskip -1pt \sqrt {  {S}_{\mathrm{off}}^2 (N,   X,   P)} + N^{-A},
	\end{align}
	with
	\begin{equation}\label{S(diag)}
		\begin{split}
			{S}_{\mathrm{diag}}^2(N, X, P) = \frac{N^{3 } X }{P^{\star 2}  P^2 K }  
			\sum_{p \shskip \sim P}
			\underset{\sstyle  (r_1 r_2,\shskip p)=1 \atop {\sstyle  |r_1 - \gamma p|, \shskip |r_2 - \gamma p| \Lt R \atop {\sstyle r_1\equiv \, r_2 (\mod p) } }}{\sum \sum}  
			\min \left\{  \frac 1 T, \frac {P N^{\sepsilon} } {KN |r_1-r_2|} \right\},
		\end{split}
	\end{equation} 
	and 
	\begin{equation}\label{S(off)}
		\begin{split}
			{S}_{\mathrm{off}}^2(N, X, P) & 
			=  \frac{N^{3 } X }{ P^{\star 2}  P^2    K  } \hskip -1pt \underset{p_1,\shskip p_2 \shskip \sim P}{\sum \sum} 
			\underset{\sstyle  (r_i ,\shskip p_i)=1 \atop {\sstyle |r_i - \gamma p_i|  \Lt R}}{\sum \sum} 
			\Bigg( \mathop{\sum_{ N/T \Lt |n| \Lt N/K }}_{n \shskip \equiv \, \overline{r}_1 p_2- \overline{r}_2 p_1      (\mod p_1p_2) } \frac {\sqrt{p_1 p_2}} {  T \hskip -1pt \sqrt{  X |n|}   } \\
			& \hskip 140pt + \mathop{\sum_{ 0 < |n| \Lt N/T }}_{n \shskip \equiv \, \overline{r}_1 p_2- \overline{r}_2 p_1      (\mod p_1p_2) } \frac 1 T \Bigg),
	\end{split}\end{equation}
	in correspondence to the cases where $n = 0$ and $n\neq 0$ in \eqref{5eq: n sum after Poisson, 2}, respectively. Since the modular form $g$ is considered fixed, we have absorbed its level $M$ into the implied constant.
	Note that in the case $n = 0$  the congruence condition in \eqref{5eq: congruence condition} would imply $p_1 = p_2$ ($=p$) and $r_1\equiv r_2 (\mod p)$. Moreover, when applying the estimates \eqref{5eq: bound for I(x)} and \eqref{5eq: bound for I(x), 2}  to $ \CalL (x) $ with $x = {MXn}/{p_1p_2}$, note that $ K^2 / T  \Lt |x| < K $ or $ |x| \Lt K^2 / T $ amounts to $ N/T \Lt |n| \Lt N/K $ or $ |n| \Lt N/T $, respectively, for $X = P^2 K^2/N$ (see \eqref{5eq: XN=P2K2}).  We record here the condition  in Lemma \ref{lem: the integral I} (2):
	\begin{align}
		\label{5eq: 2nd condtion on K}
		K > \sqrt T N^{ \sepsilon}.
	\end{align} 
	
	For $ {S}_{\mathrm{diag}}^2(N, X, P)$, we split the sum over $r_1$ and $r_2$ according as $r_1 = r_2$ or not, 
	\begin{equation*}
		{S}_{\mathrm{diag}}^2(N, X, P) = \frac{N^{3 } X }{P^{\star 2}  P^2 K }  \bigg(
		\sum_{p \shskip \sim P}
		\underset{\sstyle (r ,\shskip p)=1  \atop {\sstyle  |r -\gamma p| \Lt R }}{\sum}  \frac 1 T +  
		\sum_{p \shskip \sim P}
		\underset{\sstyle  (r_1 r_2,\shskip p)=1  \atop {\sstyle |r_1 - \gamma p|, \shskip |r_2 - \gamma p| \Lt R \atop {\sstyle r_1\equiv \, r_2 (\mod p) \atop {\sstyle r_1 \neq r_2} } }}{\sum \sum} \frac {P N^{\sepsilon} } {KN |r_1-r_2|} \bigg) ,
	\end{equation*}
	and hence
	\begin{align}\label{5eq: estimate for S diag}
		{S}_{\mathrm{diag}}^2(N, X, P) \Lt \frac{N^{3 } X }{P^{\star 2}  P^2 K } \lp \frac {P^{\star} R} T + P^{\star} R \frac { N^{\sepsilon}} {KN} \rp \Lt ( K N +  T N^{\sepsilon}) \log P. 
	\end{align}
	Recall here that $NX = P^2 K^2$ and $R = PT/N$ as in \eqref{5eq: XN=P2K2} and \eqref{5eq: R = Pt/N}.
	
	To deal with $ {S}_{\mathrm{off}}^2(N, X, P)$, we first note that necessarily $p_1\neq p_2$. Otherwise, if $p_1=p_2=p$, then the congruence $n \shskip \equiv \, \overline{r}_1 p - \overline{r}_2 p \,      (\mod p^2)$ would imply $p|n$. This is impossible, in view of our assumption $  {N^{1+\sepsilon}} / {K} < P$ in \eqref{5eq: assumption on K} and the length $  N /K$ of the $n$-sum. 
	We now interchange the sum over $n$ and the sums over $r_1$, $r_2$. Note that  for fixed $n$, the congruence  $n \equiv \overline{r}_1 p_2- \overline{r}_2 p_1       (\mod p_1p_2)$ splits into  $r_1\equiv \overline{n}  p_2  (\mod p_1)$ and $r_2\equiv -\overline{n}  p_1  (\mod p_2)$, so
	\begin{align*}
		{S}_{\mathrm{off}}^2(N, X, P) = \frac{N^{3 }X  }{ P^{\star 2}  P^2    K  } \hskip -1pt \underset{\sstyle p_1,\shskip p_2 \shskip \sim P \atop {\sstyle p_1 \neq p_2} }{\sum \sum} \Bigg( 
		{\sum_{ N/T \Lt |n| \Lt N/K }}  & \underset{\sstyle |r_1 - \gamma p_1|, \shskip |r_2 - \gamma p_2| \Lt R \atop {\sstyle r_1\equiv \shskip \overline{n}  p_2  (\mod p_1) \atop{\sstyle r_2\equiv -\overline{n}  p_1  (\mod p_2) } }}{\sum \sum} \frac {\sqrt{p_1 p_2}} {  T \hskip -1pt \sqrt{X |n|}   } \\
		+ & {\sum_{ 0 < |n| \Lt N/T }}  \underset{\sstyle |r_1 - \gamma p_1|, \shskip |r_2 - \gamma p_2| \Lt R \atop {\sstyle r_1\equiv \shskip \overline{n}  p_2  (\mod p_1) \atop{\sstyle r_2\equiv -\overline{n}  p_1  (\mod p_2) } }}{\sum \sum} \frac 1 T  \Bigg).
	\end{align*}
	When $T \geqslant N$ so that $R \geqslant P$, we have
	\begin{equation}\label{5eq: estimate for S off, 1}
		\begin{split}
			{S}_{\mathrm{off}}^2(N, X, P)   \Lt \frac{N^{3 } X }{ P^{\star 2}  P^2   K } P^{\star 2}     \lp \frac {P} {T \sqrt X} \sqrt {\frac N K} + \frac {N} {T^2} \rp \lp \frac {R}{P } \rp^2 =   \frac {N T} {\sqrt K} + K N    . 
		\end{split}
	\end{equation}
	When $T < N$, the $(R/P)^2$ in \eqref{5eq: estimate for S off, 1} needs to be replaced by $1$. In other words, we lose $(P/R)^2 = (N/T)^2$. However, the loss may be reduced to $N/T$ if we rearrange the sum ${S}_{\mathrm{off}}^2(N, X, P)$ as follows
	\begin{align*}
		\frac{N^{3 }X  }{ P^{\star 2}  P^2    K  } \hskip -1pt \sum_{p_1   \sim P} \sum_{\sstyle (r_1, p_1) = 1 \atop {\sstyle |r_1 - \gamma p_1| \Lt R}}  \Bigg( 
		{\sum_{ N/T \Lt |n| \Lt N/K }}  & \underset{\sstyle p_2 \sim P  \atop {\sstyle p_2 \equiv    \shskip  {n}  r_1  (\mod p_1)  }}{\sum}  \sum_{\sstyle |r_2-\gamma p_2| \Lt R \atop{\sstyle r_2\equiv -\overline{n}  p_1  (\mod p_2) }}  \frac {\sqrt{p_1 p_2}} {  T \hskip -1pt \sqrt{X |n|}   } \\
		+ & {\sum_{ 0 < |n| \Lt N/T }} \underset{\sstyle p_2 \sim P  \atop {\sstyle p_2 \equiv    \shskip  {n}  r_1  (\mod p_1)  }}{\sum}  \sum_{\sstyle |r_2 - \gamma p_2| \Lt R \atop{\sstyle r_2\equiv -\overline{n}  p_1  (\mod p_2) }} \frac 1 T  \Bigg).
	\end{align*}
	Thus for $T < N$, we have
	\begin{equation}\label{5eq: estimate for S off, 2}
		\begin{split}
			{S}_{\mathrm{off}}^2(N, X, P)  \Lt \frac{N^{3 } X }{ P^{\star 2}  P^2   K } P^{\star  } R     \lp \frac {P} {T \sqrt X} \sqrt {\frac N K} + \frac {N} {T^2} \rp    \Lt \lp \frac {N T} {\sqrt K} + K N \rp  \frac {N} {T} \log P  . 
		\end{split}
	\end{equation}
	Combining \eqref{5eq: estimate for S off, 1} and \eqref{5eq: estimate for S off, 2}, we have
	\begin{equation}\label{5eq: estimate for S off, 3}
		\begin{split}
			{S}_{\mathrm{off}}^2(N, X, P) \Lt  \lp \frac {N T} {\sqrt K} + K N \rp \lp 1 + \frac {N} {T} \rp \log P. 
		\end{split}
	\end{equation}

	We conclude from \eqref{5eq: S<Sdiag+Soff}, \eqref{5eq: estimate for S diag} and \eqref{5eq: estimate for S off, 3} that
	\begin{align}\label{final-bound} 
		{S} (N, X, P) \Lt  \bigg( \hskip -1pt   \hskip -1pt \sqrt T  + \lp \sqrt {KN} + \frac{\sqrt {N T} }{K^{1/4}} \rp \bigg(  1 + \sqrt { \frac {N} {T} } \bigg) \bigg) N^{\shskip \sepsilon} .
	\end{align}
	At this point a mild assumption like $P < N^{A}$ is needed so that $\log P < N^{\sepsilon}$, where $A$ is a large fixed constant.

	\subsection{Conclusion}
	In view of \eqref{4eq: S(N)=S(N,X,P)} in Proposition \ref{key lemma} and \eqref{final-bound},   we have
	\begin{equation*}
		\begin{split}
			{S(N)}  &
			\Lt     {\sqrt {T} N^{\sepsilon}} + \lp \sqrt {KN} + \frac{\sqrt {N T} }{K^{1/4}} \rp \bigg(  1 + \sqrt { \frac {N} {T} } \bigg)  N^{\shskip \sepsilon} +\frac{  N }{K} +\frac{N \sqrt {K} }{P}.
	\end{split}\end{equation*}
	For the estimate in \eqref{1eq: main bound} to be non-trivial, we assume that $ N^{\sepsilon} < T < N^{3/2 - \sepsilon} $. Then
	\begin{equation*} 
		{S(N)}  
		\Lt  T^{1/3} N^{1/2+\sepsilon} \bigg(  1 +   \frac   {N^{1/2}} {T^{1/2}}  \bigg)    + \frac{N T^{1/3} }{P}
		\Lt \, T^{1/3} N^{1/2+\sepsilon} +     \frac {N^{1+\sepsilon}} {T^{1/6} },
	\end{equation*} 
	on choosing $K = T^{2/3}$ and $P = N/T^{1/3}$. The required conditions in \eqref{5eq: XN=P2K2}, \eqref{5eq: assumption on K} and \eqref{5eq: 2nd condtion on K} are well justified for our choice of $K$ and $P$. This proves   Theorem \ref{main-theorem}.
	
	For  Corollary \ref{main-corollary}, define
	\begin{align*}
		S_H (N) =  \sum_{N \leqslant n\leqslant N+H}\lambdaup_g(n)\,e(f(n)) .
	\end{align*} 
	Let the smooth function $V$ in  Theorem \ref{main-theorem}  be supported on $[1, 1 + H/N]$ with $V (x) \equiv 1$ on $[1+1/\varDelta, 1+H/N-1/\varDelta]$. For this, it is necessary that $ \varDelta \geqslant 2 N/H  $.  By the Deligne bound \eqref{2eq: Ramanujan, 0},  we would have
	$$S_H(N)=S(N)+O(N^{1+\sepsilon}/\varDelta).$$
	Then Corollary \ref{main-corollary} follows from Theorem \ref{main-theorem} upon choosing $\varDelta = T / N^{\sepsilon}$.

	\section{Proof of the Weyl-type subconvex bound}\label{sec: Weyl}
	
	
	For $g \in S^{\star}_k (M, \xi)$  with Fourier coefficients $\lambdaup_g (n)$, let $\widebar{g} \in S^{\star}_k (M, \widebar \xi)$ be  its dual form with Fourier coefficients $\lambdaup_{\widebar{g}} (n) = \overline{\lambdaup_{g}(n)}$, and let $\epsilon_g $ be the root number of $L(s, g)$ satisfying the functional equation 
	\begin{align*}
		\Lambda (s, g) = \epsilon_g  \Lambda (1-s, \widebar g),
	\end{align*}
	with
	$$\Lambda (s, g) = 
	M^{{s} /{2}}  (2\pi)^{-s } \Gamma \Big( s + \frac {k-1} 2 \Big) L (s, g). $$
	We remark that $\epsilon_g = i^k \eta_g$ for the $\eta_g$ as in Lemma \ref{lem: Voronoi}.  From this one may deduce the following Approximate Functional Equation (see  \cite[Theorem 2.5]{Harcos02} and \cite[Lemma 2.1]{BMN}).	
	\begin{lem}[Approximate Functional Equation]
		Let $F $ be a real-valued smooth function on $(0,\infty)$ satisfying $F(x)+F(1/x)=1$ and with derivatives decaying faster than any negative power of $x$ as $x\rightarrow \infty$. Then
		\begin{equation}\label{1eq: AFE}
			\begin{split}
				L(1/2+it,g)=\sum_{n=1}^{\infty}\frac{\lambdaup_g(n)}{n^{1/2+it}} F \hskip -1pt \left(\hskip -0.5pt \frac{n}{\sqrt{C}} \hskip -0.5pt\right)+ \epsilon_g (2\pi)^{2it}\frac{\Gamma \big(\frac{k}{2}-it\big)}{\Gamma\big(\frac{k}{2}+it \big)}\sum_{n=1}^{\infty}\frac{\overline{\lambdaup_g(n)}}{n^{1/2-it}} F \hskip -1pt\left(\hskip -0.5pt \frac{n}{\sqrt{C}} \hskip -0.5pt\right) &\\
				+O_{\sepsilon,\shskip F}\big(M^{{1}/{2}} / C^{ {1}/{4}-\sepsilon}\big) &,
			\end{split}
		\end{equation} 
		where $C=C(g,t)$ is the analytic conductor defined by {\rm(\cite[(2.4)]{Harcos02})}
		$$C=\frac{M}{4 \pi^2}\left|\frac{k }{2}+it\right|\left|\frac{k}{2}+1+it\right|.$$
	\end{lem}
	
	Let $t > 1$ be large. By applying a dyadic partition of unity to the approximate functional equation \eqref{1eq: AFE}, we infer that
	\begin{equation*}
		L(1/2+it, g) \Lt   t^{\shskip \sepsilon} \lp \frac{|S(N)|}{\sqrt N } + \frac 1 {\sqrt {t \shskip} } \rp 
	\end{equation*}
	for some $N < t^{1+\sepsilon}$,  where \begin{equation*}
		\begin{split}
			S(N)=\sum_{n=1}^{\infty}\lambdaup _g(n)\,n^{-it}V\left(\frac{n}{N}\right),
	\end{split}\end{equation*}
	and $V (x)$ is some function in $C_c^{\infty}(0,\infty)$ supported on $[1, 2]$, satisfying  $V^{(j)} (x) \Lt_{j} 1$. 
	
	Recall that the Rankin--Selberg estimate in \eqref{2eq: Ramanujan} yields the trivial bound $S(N)\Lt N $. Therefore it suffices to prove the bound $S(N) \Lt \sqrt N t^{1/3+\sepsilon}$ in the range $t^{2/3+\sepsilon}<N<t^{1+\sepsilon}$.
	
	Note that $e (f(n)) = N^{it} n^{-it}$ if we choose $\phi (x) = - \log x$, $T = t/2\pi$ and $\gamma = 0$ in \eqref{0eq: f = T phi}. Consequently, Theorem \ref{main-theorem} implies that for $t^{2/3+\sepsilon}< N < t^{1+\sepsilon}$ the sum $S(N)$ has the following bound:
	\begin{align*}
		\frac {S(N)} {\sqrt N} \Lt t^{1/3+\sepsilon}, 
	\end{align*} 
	as desired.
	
	\appendix
	
	\section{Stationary phase}

	Firstly, we have Lemma {\rm8.1} in {\rm\cite{BKY-Mass}} with some improvements.
	
	\begin{lem}\label{lem: staionary phase, dim 1, 2}
		Let $\tw (x)$ be a smooth function    supported on $[ a, b]$ and $f (x)$ be a real smooth function on  $[a, b]$. Suppose that there
		are   parameters $Q, U,   Y, Z,  R > 0$ such that
		\begin{align*}
			f^{(i)} (x) \Lt_{ \, i } Y / Q^{i}, \hskip 10pt \tw^{(j)} (x) \Lt_{ \, j } Z / U^{j},
		\end{align*}
		for  $i \geqslant 2$ and $j \geqslant 0$, and
		\begin{align*}
			| f' (x) | \Gt R. 
		\end{align*}
		Then for any $A \geqslant 0$ we have
		\begin{align*}
			\int_a^b e (f(x)) \tw (x) \nd x \Lt_{ A} (b - a) Z \bigg( \frac {Y} {R^2Q^2} + \frac 1 {RQ} + \frac 1 {RU} \bigg)^A .
		\end{align*}
	\end{lem}
	
	\begin{proof}
		In the proof of Lemma {\rm8.1} in {\rm\cite{BKY-Mass}}, one can actually impose an additional condition $ \gamma_2 + \gamma_3 + ... = \varnu - n $ to the inner sum in {\rm(8.5)} so that the $Y^{  {(\varnu-\mu)} / 2}$ may be replaced by $Y^{\varnu - n}$ in {\rm(8.6)} and the sum over $\mu$ should be only up to $2 n - \varnu$. In this way,  their condition $Y \geqslant 1$ becomes unnecessary and their  estimate in {\rm(8.3)} may be improved as above.
		\delete{
			Our simple observation is that if we define $$\text{\usefont{U}{dutchcal}{m}{n} D} (f) (t) = - \frac d {d t} \bigg(\frac {f (t)}  { i h'(t) }\bigg) $$ then $\text{\usefont{U}{dutchcal}{m}{n} D}^n (f) (t) $ is a linear combination of all the terms occurring in  the  expansion of $$  \frac {d^{\shskip n}  \big( h' (t)^n f (t) \big)} {h'(t)^{2n} dt^{\shskip n } } ,$$
			where we have kept the notation from \cite{BKY-Mass}. }
	\end{proof}

	For the reader's convenience, we record here the one- and two-dimensional second derivative tests  (see  \cite[Lemma 5.1.3]{Huxley}, \cite[Lemma 4]{Munshi-Circle-III}).

	\begin{lem}\label{lem: derivative tests, dim 1}
		Let $f (x)$ be a real smooth function on  $[a, b]$. Let $\tw (x)$ be a real smooth function supported on $[ a, b]$ and let $V$ be its total variation\footnote{Since $ \tw (x)$ is supported on $[ a, b]$, we do not need to add its maximum modulus  to $V$ as in \cite[Lemma  5.1.3]{Huxley}.}.  If $ f'' (x) \geqslant \lambdaup > 0 $ on $[a, b]$, then
		\begin{align*}
			\left|\int_a^b e (f(x)) \tw (x) \nd x \right| \leqslant \frac {4 V} {\sqrt {\pi \lambdaup}}. 
		\end{align*}
	\end{lem}
	
	\begin{lem}\label{lem: 2nd derivative test, dim 2}
		Let  $f (x, y   )$ be a real smooth function on  $[a, b] \times [c, d]$ with 
		\begin{align*}
			& \left|\partial^2 f / \partial x^2 \right| \Gt \lambdaup > 0, \hskip 15pt  \left|\partial^2 f / \partial y   ^2 \right| \Gt \rho > 0, \\
			& |\det f''|  = \left|\partial^2 f / \partial x^2 \cdot \partial^2 f / \partial y   ^2 - (  \partial^2 f / \partial x \partial y    )^2 \right| \Gt \lambdaup \shskip \rho,
		\end{align*}
		on the rectangle $[a, b] \times [c, d]$.   Let $\tw (x, y   )$ be a real smooth function supported on $[ a, b]  \times [c, d]$ and let 
		\begin{align*}
			V = \int_a^b \int_c^d \left|  \frac {\partial^2 \tw(x, y   )} {\partial x \partial y   } \right| \nd x \shskip \nd y   .
		\end{align*}
		Then
		\begin{align*}
			\int_a^b \int_c^d e (f(x, y   )) \tw (x, y   ) \nd x \shskip \nd y     \Lt \frac { V } {\sqrt { \lambdaup \shskip \rho}},
		\end{align*}
		with an absolute implied constant.
	\end{lem}
	
	Finally,  
	the following stationary phase estimate is from \cite[Theorem 1.1.1]{Sogge}. 
	
	\begin{lem}\label{lem: stationary phase estimates, dim 1}
		Let $Z >0$ and $  {\lambdaup} \geqslant  1$.  	Let $\tw (x; \lambdaup)$ be a smooth function supported on $[ a, b ]$ for all $\lambdaup$, and $f (x)$ be a real smooth function on an open neighborhood of $[a, b]$. Suppose that  $  \lambdaup^{j} \partial_x^{i} \partial_\lambdaup^{  j}  \tw  (x; \lambdaup) \Lt_{ \, i, \, j } Z
		$ and that  $f(x_0) = f'(x_0) = 0$ at a point  $  x_0 \in (a, b)$, with $ f'' (x_0) \neq 0$ and $f' (x) \neq 0$ for all $x \in [a, b] \smallsetminus \{x_0\} $. Then
		\begin{align*}
			\frac {\nd^{  j}} {\nd  \lambdaup^{  j}} \int_a^b e (\lambdaup f(x)) \tw (x; \lambdaup) \, \nd x \Lt_{\, j} \frac  Z {  \lambdaup^{   1/2 +  j } }.
		\end{align*}
	\end{lem}
	
	\begin{acknowledgement}
		We would like to thank  Philippe Michel, Qingfeng Sun, and the referee for their valuable comments. 
	\end{acknowledgement}


\end{document}